\documentclass[a4paper,11pt,reqno]{amsart}

\usepackage{amsmath}%
\usepackage{amssymb}%
\usepackage{graphicx}
\usepackage{float}
\usepackage{color}
\allowdisplaybreaks

\usepackage{setspace}
\usepackage{cite}

\theoremstyle{plain}

\newtheorem{theorem}{Theorem}
\newtheorem{definition}[theorem]{Definition}
\newtheorem{example}[theorem]{Example}
\newtheorem{lemma}[theorem]{Lemma}
\newtheorem{proposition}[theorem]{Proposition}
\newtheorem{remark}[theorem]{Remark}
\newtheorem{corollary}[theorem]{Corollary}

\DeclareMathOperator\sign{sign}

\def\N{\mathbb{N}}

\def\R{\mathbb{R}}

\def\D{\mathbb{D}}

\def\ex{\mathbb{E}}

\def\eins{{\bf 1}}

\def\F{\mathcal{F}}

\def\L{\mathbb{L}}

\begin{document}

\title[Discretizing Malliavin calculus]{Discretizing Malliavin calculus}

 \author[C. Bender and P. Parczewski]{Christian Bender and Peter Parczewski}

\address{Saarland University, Department of Mathematics PO Box 151150, D-66041 Saarbr\"ucken, Germany, University of Mannheim, Institute of Mathematics A5,6, D-68131 Mannheim, Germany.}
\email{bender@math.uni-sb.de, parczewski@math.uni-mannheim.de}

\date{\today}

\begin{abstract}
Suppose $B$ is a Brownian motion and $B^n$ is an approximating sequence of rescaled random walks on 
the same probability space converging to $B$ pointwise in probability. We provide necessary and sufficient conditions for weak and strong $L^2$-convergence of a discretized Malliavin derivative, a discrete Skorokhod integral,
and discrete analogues of the Clark-Ocone derivative to their continuous counterparts. 
Moreover, given a sequence $(X^n)$ of random variables which admit a chaos decomposition in terms of discrete multiple Wiener integrals with respect to $B^n$, we derive necessary and sufficient conditions for strong $L^2$-convergence 
 to a $\sigma(B)$-measurable random variable $X$  via convergence 
of the discrete chaos coefficients of $X^n$ to  the continuous chaos coefficients of $X$. In the special case of binary noise, our results
support the known formal analogies between Malliavin calculus on the Wiener space and Malliavin calculus on the Bernoulli space by rigorous  $L^2$-convergence results.
\end{abstract}


\maketitle

\section{Introduction}

Let $B= (B_{t})_{t \geq 0}$ be a Brownian motion on  a probability space $(\Omega, \F, P)$, 
where the $\sigma$-field $\F$ is generated by the Brownian motion and completed by null sets. 
Suppose $\xi$ is a square-integrable random variable with zero expectation and variance one.
As a
discrete counterpart of $B$ we consider, for every $n\in \mathbb{N}=\{1,2,\ldots\}$, a random walk approximation
\[
B^n_t: =\frac{1}{\sqrt{n}} \sum\limits_{i=1}^{\left\lfloor nt\right\rfloor} \xi_{i}^{n} \ , \ t \geq 0,
\]
where $(\xi_{i}^n)_{i\in\N}$ is a sequence of independent random variables which have the same distribution as $\xi$.
We assume that the approximating sequence $B^n$ converges to $B$ pointwise in probability, i.e.
\begin{align}\label{eq_Donsker_embedding}
\forall \ t\geq 0 \ :\quad \lim_{n\rightarrow \infty} B^n_t = B_t \;\textnormal{ in probability.}
\end{align}
The aim of the paper is to provide $L^2$-approximation results for some basic operators of Malliavin calculus with respect to the Brownian motion $B$
such as the chaos decomposition, the Malliavin derivative, and the Skorokhod integral by appropriate sequences of 
approximating operators 
based on the discrete time noise $(\xi^n_i)_{i\in \N}$. It turns out that in all our approximation results, the limits 
do not depend on the distribution of the discrete time noise, hence our results can be regarded as some kind of invariance 
principle for Malliavin calculus.

We briefly discuss our main convergence results in a slightly informal way:
\begin{enumerate}
 \item {\it Chaos decomposition: }The heuristic idea behind the chaos decomposition in terms of multiple Wiener integrals is to project 
 a random variable $X \in L^2(\Omega,\F,P)$ on products of the white noise $\dot B_{t_1} \cdots \dot B_{t_k}$.  This idea can be made 
 rigorous with respect to the discrete noise $(\xi^n_i)_{i \in \N}$ by considering the discrete time functions
 $$
 f_X^{n,k}(i_1,\ldots i_k)= \frac{n^{k/2}}{k!} \ex\left[X \prod_{j=1}^k \xi^n_{i_j}\right] 
 $$
 for pairwise distinct $(i_1,\ldots i_k)\in \N^k$. Our results show that, after a natural embedding as step functions into continuous time, 
 the sequence $(f_X^{n,k})_{n\in \N}$ converges strongly in $L^2([0,\infty)^k)$ to the $k$th chaos coefficient of $X$, for every 
 $k\in \N$ (Example \ref{example:Wiener_integrals}). This is a simple consequence of a general Wiener chaos limit theorem 
(Theorem \ref{thm_S_transform_convergence_2}), which provides equivalent conditions for the strong $L^2(\Omega,\F,P)$-convergence 
of a sequence of random variables $(X^n)_{n\in \N}$ (with each  $X^n$  admitting a chaos decomposition via 
multiple Wiener integrals with respect to the discrete 
time noise $(\xi^n_i)_{i\in \N}$) in terms of the chaos coefficient functions. As a corollary, this Wiener chaos limit theorem 
lifts a classical result by \cite{Surgailis} on convergence in distribution of discrete multiple Wiener integrals 
to strong $L^2(\Omega,\F,P)$-convergence (in our setting, i.e. when the limiting multiple Wiener integral is driven by a Brownian motion).
\item \emph{Malliavin derivative:} With our weak moment assumptions on the discrete time noise, we cannot define 
a discrete Malliavin derivative in terms of a polynomial chaos as in the survey paper by \cite{Gzyl} and the references therein.
Instead we introduce the discretized Malliavin derivative at time $j\in \N$ with respect to the noise $(\xi^n_i)_{i\in \N}$ by
$$
D^n_jX=\sqrt{n} \ex[\xi^n_j X|(\xi^n_i)_{i\in \N \setminus \{j\}}],
$$
which is the gradient of the best approximation in $L^2(\Omega,\F,P)$ of $X$ as a linear function in $\xi^n_j$ with $\sigma(
\xi^n_i,\; {i\in \N \setminus \{j\}})$-measurable coefficients. Theorem \ref{thm_Malliavin_derivative_convergence_1} below implies that, if $(X^n)$ converges weakly in $L^2(\Omega,\F,P)$ 
to $X$ and the sequence of discretized Malliavin derivatives $(D^n_{\lceil n \cdot\rceil} X^n)_{n\in \N}$ converges 
weakly in $L^2(\Omega\times [0,\infty))$, then $X$ belongs to the domain of the continuous Malliavin derivative and the continuous
Malliavin derivative appears as the weak $L^2(\Omega\times [0,\infty))$-limit. As the Malliavin derivative is  a closed, but discontinuous operator, this 
is the best type of approximation result which can be expected when discretizing the Malliavin derivative. Sufficient conditions for 
the strong convergence of a sequence of discretized Malliavin derivatives, which can be checked in terms of the discrete-time approximations, 
are presented in Theorems \ref{thm:D22} and \ref{thm_Malliavin_derivative_convergence}. 

\item \emph{Skorokhod integral:} Defining the discrete Skorokhod integral as the adjoint operator to the discretized Malliavin derivative 
leads to
$$
\delta^n(Z^n):=\lim_{M \rightarrow \infty } \sum_{i=1}^M \ex[Z^n_i|(\xi^n_j)_{j\in \{1,\ldots,M\}\setminus \{i\}}] \frac{\xi^n_i}{\sqrt{n}},
$$
for a suitable class of discrete time processes $Z^n$, which is in line with the Riemann-sum approximation 
for Skorokhod integrals in terms of the driving Brownian motion in \cite{NualartPardoux}.
Analogous results for the `closedness across the discretization levels' as in the case of the discretized Malliavin derivative 
and sufficient conditions for strong $L^2(\Omega,\F,P)$-convergence of a sequence of discrete Skorokhod integrals
are provided in Theorems \ref{thm:Skorokhod_convergence_1}, \ref{thm:L22} and \ref{thm:Skorokhod_convergence_2}. When restricted to predictable integrands,
the convergence results for the Skorokhod integral give rise to necessary and sufficient conditions for strong and weak 
$L^2(\Omega,\F,P)$-convergence of a sequence of discrete It\^o integrals (Theorem \ref{thm:Ito_integral}). This result can be applied 
to study different discretization schemes for the generalized Clark-Ocone derivative (which provides the integrand in the predictable representation
of a square-integrable random variable as It\^o integral with respect to the Brownian motion $B$). In this respect, Theorems \ref{thm:Clark-Ocone_1}
and \ref{thm:Clark-Ocone_2} below complement related results in the literature such as \cite{Briand_Delyon_Memin, Leao_Ohashi} and the references 
therein.
\end{enumerate}

We note that related classical semimartingale limit theorems for stochastic integrals (with adapted integrands) \cite{JMP,KP} and for 
multiple Wiener integrals \cite{Surgailis, AT, Av}, or 
robustness results for martingale representations \cite{JaMP,Briand_Delyon_Memin} are usually obtained in the framework of (or using techniques of) convergence in distribution (on the Skorokhod 
space). In contrast, we exploit that strong and weak convergence in $L^2(\Omega,\F,P)$ can be characterized in terms 
of the $S$-transform, which is an important tool in white noise analysis, see e.g. \cite{Kuo, Janson, Holden_Buch}, 
and corresponds to taking expectation under suitable changes of measure. We introduce a discrete version of the $S$-transform 
in terms of the noise $(\xi^n_i)_{i\in \N}$ and show that strong and weak $L^2(\Omega,\F,P)$-convergence can be equivalently expressed
via convergence of the discrete $S$-transform to the continuous $S$-transform (Theorem \ref{thm_S_transform_convergence_1}).
With this observation at hand, all our convergence results can be obtained in a surprisingly simple way by  computing suitable $L^2(\Omega, \sigma(\xi^n_i)_{i\in \N}, P)$-inner products 
and their limits as $n$ tends to infinity.
However, all these results can be seen as strong and weak invariance principles for Malliavin calculus.

The paper is organized as follows: In Section \ref{section_s_transform}, we introduce the discrete $S$-transform 
and discuss the connections between weak (and strong) $L^2(\Omega,\F,P)$-convergence and the convergence of the discrete $S$-transform 
to the continuous one. Equivalent conditions for the weak $L^2$-convergence of sequences of discretized Malliavin derivatives 
and discrete Skorokhod integrals to their continuous counterparts are derived in Section \ref{section_Malliavin_weak}. Combining these weak 
$L^2$-convergence results with the duality between discrete Skorokhod integral and discretized Malliavin derivative, we also identify sufficient 
conditions for the strong $L^2$-convergence which can be checked solely in terms of the discrete time approximations. We are not aware of 
any such convergence results for general discrete time noise distributions in the literature. In Section \ref{section:Ito},
we specialize to the nonanticipating case and prove limit theorems for discrete It\^o integrals and discretized Clark-Ocone derivatives.
The strong $L^2$-Wiener chaos limit theorem is presented in Section \ref{section_chaos_decomposition}, and is applied 
in order to provide equivalent conditions for the strong $L^2$-convergence of sequences of discretized Malliavin derivatives 
and discrete Skorokhod integrals in terms of tail conditions of the discrete chaos coefficients in Section \ref{sec:malliavin2}. Finally, in Section \ref{section_binary}, we consider the special case
of binary noise, in which discrete Malliavin calculus is very well studied, see e.g. the monograph by \cite{Privault_book}. 
We explain that the statement of our convergence results can be simplified in this case and demonstrate by a toy example 
how to apply the results numerically in a Monte Carlo framework.

\section{Weak and strong $L^2$-convergence  via  discrete S-transforms}\label{section_s_transform}



In this section, we study strong and weak $L^2(\Omega,\F,P)$-convergence of a sequence $(X^n)$ of random variables, where $X^n$ is 
$\mathcal{F}^n:=\sigma(\xi_i^n,\;i\in\N)$-measurable,
to an $\mathcal{F}$-measurable $X$. As a main result of this section (Theorem \ref{thm_S_transform_convergence_1}), we provide an equivalent criterion for this convergence, 
which only requires to compute a family of $L^2(\Omega,\mathcal{F}^n,P)$-inner products (hence, expectations which involve functionals 
of the discrete time noise $(\xi^n_i)_{i\in \N}$ only) and their limits as $n$ tends to infinity.

Before doing so, let us recall that  $B^n$ can be constructed via a Skorokhod embedding of the random walk 
 $$
 \left(\sum\limits_{i=1}^{j} \xi_{i}\right)_{j\in \N},\quad \xi_1,\xi_2,\ldots  \textnormal{ independent and with the same distribution as }\xi, 
 $$
 into the rescaled Brownian motion $(\sqrt{n} B_{t/n})_{t\geq 0}$. In this way, one obtains, for every $n\in \N$, a sequence 
 of stopping times $(\tau^n_i)_{i\in \N_0}$ with respect to the augmentation of the filtration generated by $B$ such that
 \begin{equation}\label{eq:Skorokhod_embedding}
 B^n := \left(B_{\tau^n_{\left\lfloor nt\right\rfloor}}\right)_{t\geq 0}
 \end{equation}
 has the same distribution as $(\frac{1}{\sqrt{n}} \sum\limits_{i=1}^{\left\lfloor nt\right\rfloor} \xi_{i} )_{t\geq 0}$ and converges to $B$  uniformly on compacts in probability (see e.g. \cite[Lemma 5.24 (b)]{Moerters_Peres}). 

We now introduce the $S$-transform simultaneously in the continuous time setting  and the discrete time setting, which turns ou to be the key 
tool for the proofs of our limit theorems.
Recall, that the mapping $\eins_{(0,t]}\mapsto B_t$ can be extended to a continuous linear mapping from 
$L^2([0,\infty))$ to $L^2(\Omega,\mathcal{F},P)$, which is known as the \emph{Wiener integral}. We denote
the Wiener integral of a function $f\in L^2([0,\infty))$ by $I(f)$.
The \emph{discrete Wiener integral} is given by
$$
I^{n}(f^{n}) := \frac{1}{\sqrt{n}}\sum\limits_{i=1}^{\infty} f^n(i) \xi_{i}^{n}.
$$
Here, the discrete time function $f^n$ is a member of
\[
L^2_n(\N) := \left\{f^n: \N \rightarrow \R : \ \|f^n\|_{L^2_n(\N)}^2 := \frac{1}{n}\sum_{i=1}^{\infty} (f^n(i))^2<\infty\right\},
\]
which obviously ensures that the series $I^{n}(f^{n})$ converges (strongly) in $L^2(\Omega,\mathcal{F}^n,P)$.

 The \emph{Wick exponential} is, by definition, the stochastic exponential of a Wiener integral $I(f)$, i.e.,
\begin{equation*}
\exp^{\diamond}(I(f)) := \exp\left(I(f) - 1/2\int_{0}^{\infty}f^{2}(s) ds\right). 
\end{equation*}
Hence, its discrete counterpart, the \emph{discrete Wick exponential}, is given by
\begin{equation}\label{eq_discrete_Wick_exp}
\exp^{\diamond_n}(I^n(f^n)) := \prod\limits_{i=1}^{\infty}\left(1+ \frac{1}{\sqrt{n}} f^n(i)\xi_{i}^{n}\right).
\end{equation}
In particular, by Fatou's lemma and the estimate $1+x \leq \exp(x)$,
\begin{equation}\label{eq:exponential_L2}
\ex[(\exp^{\diamond_n}(I^n(f^n)))^2]\leq   \exp(\|f^n\|_{L^2_n(\N)}^2)<\infty.
\end{equation}
Notice also that
\begin{eqnarray}\label{eq:Doleans_Dade_discrete}
  \nonumber \exp^{\diamond_n}(I^n( f^n))&=&1+\sum_{i=1}^\infty \left(\exp^{\diamond_n}(I^n(f^n\eins_{[1,i]})) - \exp^{\diamond_n}(I^n(f^n\eins_{[1,i-1]})) \right)
  \\ &=&1+\sum_{i=1}^\infty f^n(i) \exp^{\diamond_n}(I^n(f^n\eins_{[1,i-1]})) \frac{\xi^n_i}{\sqrt{n}},
 \end{eqnarray}
 which is the discrete counterpart of the Dol\'eans-Dade equation.

 We finally recall that, for every $X \in L^2(\Omega, \F, P)$ and $f \in L^2([0,\infty))$, the \emph{S-transform} is defined as
\[
(SX)(f) := \ex[X \exp^{\diamond}(I(f))]. 
\]
Analogously, for every $X^n \in L^2(\Omega, \F^n, P)$ and $f^n \in L^2_n(\N)$, we introduce the \emph{discrete S-transform} as
\[
(S^nX^n)(f^n) := \ex[X^n \exp^{\diamond_n}(I^n(f^n))]. 
\]
We emphasize that the $S$-transform is a powerful tool in the white noise analysis, see, e.g., \cite{Kuo},
and has been succesfully applied in the theory of stochastic partial differential equations, see \cite{Holden_Buch}.
To the best 
of our knowledge the discrete $S$-transform has, however, not been studied in the literature.

Let us next denote by 
 $\mathcal{E}$ the set of step functions on left half-open intervals, i.e., functions of the form
 $$
 g(x)= \sum\limits_{j=1}^{m} a_j \eins_{(b_j,c_j]}(x),\quad m\in \N, a_j,b_j, c_j\in \R.
 $$
As the set of Wick exponentials of step functions $\{\exp^{\diamond}(I(g)),\; g\in \mathcal{E}\}$ is total in $L^2(\Omega,\mathcal{F},P)$,
see e.g. \cite[Corollary 3.40]{Janson},
every $L^2(\Omega,\F,P)$-random variable is uniquely determined by its $S$-transform. More precisely, if for $X,Y\in L^2(\Omega,\mathcal{F},P)$,
$
(SX)(g)=(SY)(g)
$
for every $g\in \mathcal{E}$, then $X=Y$ $P$-almost surely. We define the discretization of a step function $g\in \mathcal{E}$  as
\[
\check{g}^n = (\check{g}^n(1), \check{g}^n(2), \ldots) := \left(g(1/n), g(2/n), \ldots\right),
\]
and notice that 
\begin{equation*}
\{\check{g}^n : g \in \mathcal{E}\} \subset L^2_n(\N) 
\end{equation*}
is the dense  subspace of discrete time functions with finite support. 

The convergence results of integral and derivative operators in this paper rely on the following characterization of $L^2(\Omega,\F,P)$-convergence in terms of
convergence of the discrete $S$-transform to the continuous $S$-transform.
\begin{theorem}\label{thm_S_transform_convergence_1}
 Suppose $X, X^n \in L^2(\Omega, \F, P)$ for every $n\in\N$, with $X^n$ being $\F^n$-measurable. Then the following assertions are equivalent as $n$ tends to infinity:
 \begin{enumerate}
  \item[(i)] $X^n \rightarrow X$   strongly (resp. weakly) in $L^2(\Omega, \F, P)$.
    \item[(ii)] 
  $(S^n X^n)(\check{g}^n) \rightarrow (S X)(g)$ for every $g\in \mathcal{E}$, and additionally  $\ex[(X^n)^2]  \rightarrow \ex[X^2]$ in the case of strong convergence 
  (resp. $\sup_{n\in \N} \ex[(X^n)^2]<\infty$ in the case of weak convergence).
 \end{enumerate}
 Moreover, in the case of strong convergence, (i) is also equivalent to
 \begin{enumerate}
  \item[(iii)]$\left(X^n, \exp^{\diamond_n}(I^n(\check{g}^n))\right)  {\rightarrow} \left(X, \exp^{\diamond}(I(g))\right)$ in 
  distribution for every $g \in \mathcal{E}$, and $((X^n)^2)_{n \in \N}$ is uniformly integrable.
 \end{enumerate}

 \end{theorem}

 \begin{remark}\label{rem:S_convergence}
  Note, that $X\in L^2(\Omega,\mathcal{F},P)$ is, of course, not determined by its univariate distribution, but it is uniquely 
  determined by all the bivariate distributions of $(X, e^{\diamond I(g)})$,  $g\in \mathcal{E}$, in view of the 
  injectivity of the $S$-transform. This observation motivates that 
  the characterization of strong $L^2(\Omega,\mathcal{F},P)$-convergence via convergence in distribution in item (iii) of Theorem \ref{thm_S_transform_convergence_1}
  can hold.
 \end{remark}

In view of Lemma \ref{lem:weakconvergence} below, the proof of Theorem \ref{thm_S_transform_convergence_1} can be reduced to the following strong $L^2$-convergence result for 
(discrete) Wick exponentials.

\begin{proposition}\label{proposition_Wiener_and_Wick_exp}
Suppose $g \in \mathcal{E}$. Then,  we have strongly in $L^2(\Omega,\F,P)$, as $n$ tends to infinity:
\[
\exp^{\diamond_n}(I^n(\check{g}^n)) \rightarrow \exp^{\diamond}(I(g)). 
\]
\end{proposition}
These type of convergence results for stochastic exponentials are somewhat standard and can be obtained in a much more
general context by applying weak convergence results for stochastic differential equations, see, e.g., \cite{Av, KP} and the references
therein.
For sake of completeness, we here provide an elementary proof.

\begin{proof}
Let
\[
g = \sum\limits_{j=1}^{m} a_j \eins_{(b_j,c_j]} \in \mathcal{E}.
\]
We denote by $C,N$ constants in $\N$ such that $g$ is bounded by $C$ and has support in $[0,N]$.
Decomposing
\begin{eqnarray*}
&& \ex\left[\left(\exp^{\diamond}(I(g)) -\exp^{\diamond_n}(I^n(\check{g}^n))\right)^2\right]
 \\ &=&  \ex\left[\left(\exp^{\diamond}(I(g))\right)^2\right]- 2\ex\left[\exp^{\diamond_n}(I^n(\check{g}^n))\exp^{\diamond}(I(g))  \right] 
  + \ex\left[\left(\exp^{\diamond_n}(I^n(\check{g}^n)) \right)^2\right],
\end{eqnarray*}
it suffices to show
\begin{enumerate}
 \item[(i)] $\lim_{n\rightarrow \infty} \ex\left[\left(\exp^{\diamond_n}(I^n(\check{g}^n)) \right)^2\right]= \ex\left[\left(\exp^{\diamond}(I(g))\right)^2\right]$,
 \item[(ii)] $\exp^{\diamond_n}(I^n(\check{g}^n)) \rightarrow \exp^{\diamond}(I(g))$ in probability,
\end{enumerate}
because under (i) the integrand in the second term on the right-hand side is uniformly integrable.
\\[0.1cm]
(i) 
Due to $p < \lceil q\rceil \leq r \Leftrightarrow \lfloor p \rfloor <q \leq \lfloor r \rfloor$ for all $p,q,r \in \R$, we obtain for every $t\in(0,\infty)$,
\begin{eqnarray}\label{eq:checkg}
 \check{g}^n(\lceil nt\rceil)&=&\sum\limits_{j=1}^{m} a_j \eins_{(\lfloor b_jn\rfloor/n,\lfloor c_jn\rfloor/n]}(t).
\end{eqnarray}
Hence,
\begin{equation}\label{eq:hilf0001}
\|g- \check{g}^n(\lceil n \cdot\rceil)\|_{L^2([0,\infty))}\leq \sqrt{2} \left(\sum\limits_{j=1}^{m} |a_j| \right) \frac{1}{\sqrt{n}}\rightarrow 0,
\end{equation}
and in particular,
$$
\sum_{i=1}^{Nn} \left(\check g^n(i)\right)^2\frac{1}{n}=\| \check{g}^n(\lceil n \cdot\rceil)\|^2_{L^2([0,\infty))}
\rightarrow  \int_0^\infty g(s)^2 ds.
$$
Thus, by the independence of the centered random variables $(\xi^n_i)_{i\in \N}$ with unit variance and taking the boundedness 
of $g$ into account, we get
\begin{eqnarray*}
 \ex\left[(\exp^{\diamond_n}(I^n(\check{g}^n)))^2\right] &=& \prod_{i=1}^{Nn} 
 \ex\left[(1+\frac{1}{\sqrt{n}} \check{g}^n(i) \xi^n_i)^2\right]
 =\prod_{i=1}^{Nn} \left(1+ \frac{1}{n} (\check{g}^n(i))^2 \right)\nonumber\\ 
&\rightarrow& \exp\left(\int_0^\infty g(s)^2 ds \right)=\ex\left[\left(\exp^{\diamond}(I(g))\right)^2\right].
\end{eqnarray*}
(ii) In order to treat the large jumps of $B^n$ and the small ones separately, we consider
$$
\xi^{n,1}_i:=\xi^n_i\eins_{\{|\xi^n_i|\leq \frac{\sqrt{n}}{2C} \}},\quad \xi^{n,2}_i:=\xi^n_i\eins_{\{|\xi^n_i|> \frac{\sqrt{n}}{2C} \}},
$$
cp. also \cite{Sottinen}. Then,
$$
\exp^{\diamond_n}(I^n(\check{g}^n)) =\prod\limits_{i=1}^{Nn}\left(1+ \frac{1}{\sqrt{n}} \check g^n(i)\xi_{i}^{n,1}\right)
\prod\limits_{i=1}^{Nn}\left(1+ \frac{1}{\sqrt{n}} \check g^n(i)\xi_{i}^{n,2}\right)=: E^{n,1}\cdot E^{n,2} 
$$
We note that, for every $\epsilon>0$, by the independence of $(\xi^n_i)_{i \in \N}$,
\begin{equation}\label{eq:sup_probab}
 P\left(\left\{\sup_{i=1,\ldots, Nn} \frac{|\xi^n_i|}{\sqrt{n}}>\epsilon\right\}\right)=1-\left(1- \frac{P(\{|\xi|>\epsilon \sqrt{n}\})Nn}{Nn} \right)^{Nn}\rightarrow 0,
\end{equation}
because, by square-integrability of $\xi$, $P(\{|\xi|>\epsilon \sqrt{n}\})n\rightarrow 0$, see, e.g., \cite[p. 208]{Sh}.

Hence, for every $\epsilon>0$,
$$
P(\{|E^{n,2}-1|\geq \epsilon\})\leq P(\{\sup_{i=1,\ldots, Nn} |\xi^{n,2}_i|>0\}) = P\left(\left\{\sup_{i=1,\ldots, Nn} \frac{|\xi^n_i|}{\sqrt{n}}>1/(2C) \right\}\right) \rightarrow 0,
$$
i.e., $(E^{n,2})_{n\in \N}$ converges to 1 in probability.

By construction, each factor in $E^{n,1}$ is larger than $1/2$. Applying a Taylor expansion to the logarithm, thus, yields

$$
\log  E^{n,1} = \sum_{i=1}^{Nn} \check g^n(i) \frac{\xi^{n,1}_i}{\sqrt{n}}-\frac{1}{2}\sum_{i=1}^{Nn} \left(\check g^n(i)\right)^2\frac{(\xi^{n,1}_i)^2}{n} + R_n
$$
with a remainder term satisfying
$$
|R_n|\leq \frac{8C}{3} \left( \sup_{j=1,\ldots, Nn} \frac{|\xi^n_j|}{\sqrt{n}}\right) \sum_{i=1}^{Nn} \left(\check g^n(i)\right)^2\frac{(\xi^{n,1}_i)^2}{n}.
$$
It, thus, suffices to show
\begin{enumerate}
 \item[(iii)] $\sum_{i=1}^{Nn} \check g^n(i) \frac{\xi^{n,1}_i}{\sqrt{n}}\rightarrow I(g)$ in probability,
 \item[(iv)] $\sum_{i=1}^{Nn} \left(\check g^n(i)\right)^2\frac{(\xi^{n,1}_i)^2}{n}\rightarrow \int_0^\infty g(s)^2 ds$ in probability.
\end{enumerate}
Indeed, by (\ref{eq:sup_probab}), the remainder term then vanishes in probability as $n$ tends to infinity, and, thus,
$$
E^{n,1}\rightarrow \exp\left( I(g)- \frac{1}{2}\int_0^\infty g(s)^2 ds \right) \textnormal{ in probability.}
$$
The same argument, which was applied for the convergence of $E^{n,2}$, shows that we can (and shall) replace $\xi^{n,1}_i$ by $\xi^n_i$ in (iii)
and (iv). However, by (\ref{eq_Donsker_embedding}) and (\ref{eq:checkg}),
$$
\lim_{n\rightarrow \infty} \sum_{i=1}^{Nn} \check g^n(i) \frac{\xi^{n}_i}{\sqrt{n}}=\lim_{n\rightarrow \infty} \sum\limits_{j=1}^{m} a_j\left(B^n_{c_j}-B^n_{b_j}\right)= 
 \sum\limits_{j=1}^{m} a_j\left(B_{c_j}-B_{b_j}\right)=I(g), \textnormal{ in probability.}
$$
Finally, by the law of large numbers, $\frac{1}{n} \sum_{i=1}^{\lfloor nt \rfloor} (\xi^{n}_i)^2$ converges to $t$ in probability 
for every $t\geq 0$, and, hence, by  (\ref{eq:checkg}),
$$
\lim_{n\rightarrow \infty} \sum_{i=1}^{Nn} \left(\check g^n(i)\right)^2\frac{(\xi^{n}_i)^2}{n}=\sum\limits_{j=1}^{m} a_j^2\left({c_j}-{b_j}\right)=
\int_0^\infty g(s)^2 ds, \textnormal{ in probability.}
$$
\end{proof}

The following simple lemma from functional analysis turns out to be useful. 

\begin{lemma}\label{lem:weakconvergence}
 Suppose $H$ is a Hilbert space, $A$ is an arbitrary index set, $\{x^a,\;a \in A\}$ is total in $H$, and, for every $a\in A$,
 $(x^a_n)_{n\in \N}$ is a sequence in  $H$ which converges strongly in $H$ to $x^a$. Then, the following are equivalent, as $n$ tends to infinity:
 \begin{enumerate}
  \item[(i)] $x^n \rightarrow x$  strongly (resp. weakly) in $H$.
    \item[(ii)]
  $\langle x_n, x^a_n\rangle_H \rightarrow \langle x, x^a\rangle_H$ for every  $a \in A$,  and additionally  $\|x_n\|_H \rightarrow \|x\|_H$ in the case of strong convergence 
  (resp. $\sup_{n\in \N} \|x_n\|_H<\infty$ in the case of weak convergence).
 \end{enumerate}
\end{lemma}

\begin{proof}
 Firstly, we observe that $\sup_{n\in \N} \|x_n\|_H$ is finite, either by weak convergence \cite[Theorem V.1.1]{Yosida} in (i) or by assumption (ii). Thus, for every $a\in A$, by the strong convergence of $(x^a_n)$ to $x^a$,
 \begin{equation}\label{eq:hilf0018}
 \left|\langle x_n, x^a_n\rangle_H- \langle x_n, x^a\rangle_H\right| =|\langle x_n, x^a_n-x^a\rangle_H|\leq \sup_{m\in \N} \|x_m\|_H \|x^a_n-x^a\|_H \rightarrow 0.
 \end{equation}
Let us treat the case of weak convergence: If (i) holds, the term $\langle x_n, x^a\rangle_H$ in (\ref{eq:hilf0018}) 
converges to $\langle x, x^a\rangle_H$, and then so does $\langle x_n, x^a_n\rangle_H$, which implies (ii). Conversely, if (ii) holds,  
the first term $\langle x_n, x^a_n\rangle_H$ in (\ref{eq:hilf0018}) tends to $\langle x, x^a\rangle_H$,
and then so does $\langle x_n, x^a\rangle_H$, which yields (i) in view of \cite[Theorem V.1.3]{Yosida}. 
The case of strong convergence is an immediate consequence, as, in a Hilbert space, strong convergence is equivalent 
to weak convergence and convergence of the norms  \cite[Theorem V.1.8]{Yosida}.
\end{proof}

We are now in the position to prove Theorem \ref{thm_S_transform_convergence_1}.

\begin{proof}[Proof of Theorem \ref{thm_S_transform_convergence_1}]
`$(i) \Leftrightarrow (ii)$': Proposition  \ref{proposition_Wiener_and_Wick_exp} 
and Lemma \ref{lem:weakconvergence} apply immediately in view of the definition of the (discrete) $S$-transform, and 
as the set of Wick exponentials of step functions $\{\exp^{\diamond}(I(g)),\; g\in \mathcal{E}\}$ is total in $L^2(\Omega,\mathcal{F},P)$.\\
`$(i)$ with strong convergence $\Rightarrow (iii)$': This is is a direct consequence
of Proposition \ref{proposition_Wiener_and_Wick_exp} and the assumed strong $L^2(\Omega,\mathcal{F},P)$-convergence
of $(X_n)$. \\ 
`$(iii) \Rightarrow (ii)$ with strong convergence': 
 By (iii) and the continuous mapping theorem, 
the sequence  
$(X^n \exp^{\diamond_n}(I^n(\check{g}^n)))$ converges in distribution to $X \exp^{\diamond}(I({g}))$. 
 Moreover, this sequence is uniformly integrable, because so are the sequences
 $(|X^n|^2)$ by assumption and $( |\exp^{\diamond_n}(I^n(\check{g}^n))|^2)$ by
Proposition \ref{proposition_Wiener_and_Wick_exp}. Hence,
$$
(S^n X^n)(\check{g}^n) = \ex[X^n \exp^{\diamond_n}(I^n(\check{g}^n)]\rightarrow \ex[X \exp^{\diamond}I(g)]=(SX)(g).
$$
Moreover, thanks to the uniform integrability of  $((X^n)^2)$ and the convergence in distribution 
$X^n \stackrel{d}{\rightarrow} X$, we have $\ex[(X^n)^2]  \rightarrow \ex[X^2]$. This completes the proof of $(ii)$ with strong convergence.
\end{proof}

We close this section with an example.

\begin{example}\label{ex:conditional}
 (i) In this example, we provide a simple proof, that, for every $X\in L^2(\Omega,\mathcal{F},P)$, $X^n:=\ex[X|\mathcal{F}^n]$
 converges to $X$ strongly in 
 $L^2(\Omega,\mathcal{F},P)$. Indeed, by Proposition \ref{proposition_Wiener_and_Wick_exp}, for every $g\in \mathcal{E}$,
 $$
 (S^n X^n)(\check{g}^n) =\ex\left[\ex[X|\mathcal{F}^n] \exp^{\diamond_n}(I^n(\check{g}^n))\right]=\ex\left[X \exp^{\diamond_n}(I^n(\check{g}^n))\right] \rightarrow \ex\left[X \exp^{\diamond}(I(g))\right]=(S X)(g).
 $$
 As  $\ex[(X^n)^2]\leq \ex[X^2]$, Theorem \ref{thm_S_transform_convergence_1}  implies weak $L^2(\Omega,\mathcal{F},P)$-convergence of
 $(X^n)$ to $X$. The same theorem  finally yields strong  $L^2(\Omega,\mathcal{F},P)$-convergence,
 since, by the already established weak convergence,
 $$
 \ex[(X^n)^2]= \ex\left[\ex\left[\left.X^n \right|\mathcal{F}^n \right] X \right]  = \ex[X^nX]\rightarrow \ex[X^2].
 $$
 We note that this result can alternatively be derived by the uniform integrability of $((X^n)^2)$ via the concept 
 of convergence of filtrations making use of \cite[Proposition 2]{CMS}.
 \\[0.1cm] (ii) Denote by $(\mathcal{F}_t)_{t\geq 0}$ the augmented Brownian filtration and let $\mathcal{F}^n_i=\sigma(\xi^n_1,\ldots, \xi^n_i)$.
 We assume $X \in L^2(\Omega,\F_T,P)$. 
Then, one can always approximate $X$ by a sequence $(X^n_T)$ strongly in $L^2(\Omega,\mathcal{F},P)$, where $X^n_T$ is  measurable with respect
to $\F^{n}_{{\lfloor nT\rfloor}}$. Indeed, take any sequence $(X^n)$ of $\mathcal{F}^n$-measurable
random variables which converges strongly in $L^2(\Omega,\mathcal{F},P)$ to $X$, and define $X^n_T=\ex[X^n|\mathcal{F}^n_{{\lfloor nT\rfloor}}]$. Then, for every $g\in \mathcal{E}$,
by Proposition \ref{proposition_Wiener_and_Wick_exp},
\begin{eqnarray*}
 (S^nX^n_T)(\check g^n)&=& 
 \ex\left[X^n \prod\limits_{i=1}^{\lfloor nT \rfloor}\left(1+ \frac{1}{\sqrt{n}} g(i/n)\xi_{i}^{n}\right) \right]
 =
 \ex\left[X^n \exp^{\diamond_n}(I^n(\check{(g\eins_{(0,T]})}^n))\right]\\ &\rightarrow&    
 \ex\left[X \exp^\diamond(I(g\eins_{(0,T]})) \right]= \ex\left[X \ex[\exp^\diamond(I(g))|\F_T] \right]=(SX)\left(g\right).
\end{eqnarray*}
Moreover, 
$$
\sup_{n\in \N} \ex[(X^n_T)^2]\leq \sup_{n\in \N} \ex[(X^n)^2]<\infty.
$$
Hence, $(X^n_T)$ converges weakly in $L^2(\Omega,\mathcal{F},P)$ to $X$ by Theorem \ref{thm_S_transform_convergence_1}. Then, strong $L^2(\Omega,\mathcal{F},P)$-convergence 
follows by Theorem \ref{thm_S_transform_convergence_1} as well, because
$$
\ex[(X^n_T)^2]=\ex[X^n_TX]+\ex[X^n_T(X^n-X)]\rightarrow \ex[X^2].
$$

\end{example}

\section{Weak $L^2$-approximation of the Skorokhod integral and the Malliavin derivative}\label{section_Malliavin_weak}

In this section, we first discuss weak $L^2$-approximations of the Skorokhod integral and the Malliavin derivative
via appropriate discrete-time counterparts. We then show how to lift these results from weak convergence to strong convergence 
via duality under appropriate conditions which can be formulated in terms of the discrete-time approximations.

While most presentations of Malliavin calculus first introduce the Malliavin derivative and then define the Skorokhod integral 
as adjoint operator of the Malliavin derivative, we shall here employ the following equivalent characterization of the Skorokhod integral 
in terms of the $S$-transform, cp. \cite[Theorem 16.46, Theorem 16.50]{Janson}.

\begin{definition}\label{def:Skorokhod}
 $Z\in L^2(\Omega\times [0,\infty)):= L^2(\Omega \times [0,\infty), \F \otimes \mathcal{B}([0,\infty)), P \otimes \lambda_{[0,\infty)})$ is said to belong to the domain $D(\delta)$
 of the  Skorokhod integral, if there is an $X\in L^2(\Omega,\F,P)$ such that for every $g\in \mathcal{E}$
 $$
(SX)(g)=\int_0^\infty (SZ_t)(g)g(t)dt.
 $$
 In this case, $X$ is uniquely determined and $\delta(Z):=X$ is called the \emph{Skorokhod integral of $Z$}.
\end{definition}

For the discrete-time approximation we first introduce the space 
\begin{align*}
L^2_{n}(\Omega \times \N) &:= \left\{Z^n:\N\rightarrow L^2(\Omega,\F^n,P),\; \|Z^n\|^2_{L^2_{n}(\Omega \times \N)}:= \frac{1}{n} \sum_{i=1}^\infty \ex[(Z^n_i)^2]<\infty\right\}.
\end{align*}
Moreover, we recall the definitions
$$
\F^n := \sigma(\xi^n_j, j \in \N), \qquad \F^n_{M} := \sigma(\xi^n_1, \ldots, \xi_M),
$$
and introduce the shorthand notations
$$
\F^n_{-i} := \sigma(\xi^n_j, j \in \N\setminus \{i\}), \qquad \F^n_{M,-i} := \sigma(\xi^n_j, j \in \{1,\ldots, M\}\setminus \{i\}).
$$

\begin{definition}
 We say, $Z^n \in L^2_{n}(\Omega \times \N)$ belongs to the domain $D(\delta^n)$ of the discrete Skorokhod integral, if 
 \begin{equation}\label{eq:discrete_Skorokhod}
 \delta^n(Z^n):=\lim_{M \rightarrow \infty } \sum_{i=1}^M \ex[Z^n_i|\F^n_{M,-i}] \frac{\xi^n_i}{\sqrt{n}}.
\end{equation}
exists strongly in $L^2(\Omega,\F,P)$. If this is the case, $\delta^n(Z^n)$ is called the \emph{discrete Skorokhod integral} 
of $Z^n$.
\end{definition}
We note that, by the independence of $\ex[Z^n_i|\F^n_{M,-i}]$ and $\xi^n_i$, each summand 
on the right-hand side of (\ref{eq:discrete_Skorokhod}) is indeed a member of $L^2(\Omega,\F,P)$. Moreover, the martingale convergence 
theorem implies that, for every  $Z^n \in L^2_{n}(\Omega \times \N)$ and $N\in\N$, $Z^n \eins_{[1,N]} \in D(\delta^{n})$
and
\begin{equation}\label{eq:discrete_Skorokhod_finite}
\delta^n(Z^n \eins_{[1,N]})=\sum_{i=1}^N \ex[Z^n_i|\F^n_{-i}] \frac{\xi^n_i}{\sqrt{n}}.
\end{equation}
Hence, the discrete Skorokhod integral is densely defined from $L^2_n(\Omega\times \N)$ to $L^2(\Omega,\F,P)$. 
We will show in Proposition \ref{prop:duality_discrete} below that it is a closed operator.

\begin{remark}
 This definition of the discrete Skorokhod integral closely resembles the following Riemann-sum approximation of the 
 Skorokhod integral by \cite{NualartPardoux}, who show that under appropriate conditions on $Z$,
 $$
 \delta(Z\eins_{[0,1]})=\lim_{n\rightarrow \infty} \sum_{i=0}^{n-1} \ex\left[\left.n\int_{i/n}^{(i+1)/n} Z_s ds \right|(B_s,B_1-B_{r})_{0\leq s\leq i/n\leq (i+1)/n\leq r\leq 1} \right] \left(B_{(i+1)/n}-B_{i/n}\right)
 $$
 strongly in $L^2(\Omega,\F,P)$.
\end{remark}

As a first main result of this section we are going to show the following weak approximation theorem for Skorokhod integrals.
\begin{theorem}\label{thm:Skorokhod_convergence_1}
Suppose $Z^n\in D(\delta^n)$ for every $n\in \N$, and $(Z^n_{\lceil n\cdot\rceil})_{n\in\N}$ converges to $Z$ weakly in $L^2(\Omega\times [0,\infty))$.
Then, the following assertions are equivalent:
\begin{itemize}
 \item[(i)]  $\sup_{n\in \N} \ex[|\delta^n(Z^n)|^2 ]<\infty$.
 \item[(ii)] $Z\in D(\delta)$ and  $(\delta^n(Z^n))$ converges  to $\delta(Z)$ weakly in $L^2(\Omega,\F,P)$ as $n$ tends to infinity.
 \end{itemize}
\end{theorem}

As a first tool for the proof we state the discrete $S$-transform of a discrete Skorokhod integral.
\begin{proposition}\label{prop:Skorokhod_Stransform_discrete}
 Suppose $Z^n \in D(\delta^n)$. Then, for every $g\in \mathcal{E}$,
 $$
\left( S^n \delta^n(Z^n)\right)(\check g^n)=\frac{1}{n} \sum_{i=1}^\infty (S^nZ^n_i)(\check{g}^n\eins_{\N\setminus\{i\}})\check{g}^n(i).
 $$
\end{proposition}
This result is a special case of the more general Proposition \ref{prop:duality_discrete} below, to which we refer the reader for the proof.

The second tool for the proof of Theorem \ref{thm:Skorokhod_convergence_1} is the following variant of Theorem \ref{thm_S_transform_convergence_1}
for stochastic processes.
\begin{theorem}\label{thm_S_transform_process_convergence_discrete}
Suppose $Z \in  L^2(\Omega \times [0,\infty))$, $(Z^n)_{n\in \N}$ satisfies $Z^n \in L^2_n(\Omega \times \N)$ for every $n\in \N$. Then the following assertions are equivalent as $n$ tends to infinity:
\begin{enumerate}
\item[(i)] $(Z^n_{\left\lceil n \cdot\right\rceil})$ converges strongly (resp. weakly) to $Z$  in $L^2(\Omega \times [0,\infty))$.
\item[(ii)] For every $g, h \in \mathcal{E} $
$$
 \frac{1}{n}\sum_{i=1}^{\infty} (S^n Z^n_i)(\check{g}^n)\check{h}^n(i) \rightarrow \int_{0}^{\infty}(S Z_s)(g) h(s) ds.
$$
and, additionally, $\ex[\int_{0}^{\infty} (Z^n_{\left\lceil ns\right\rceil})^2 ds] \rightarrow \ex[\int_{0}^{\infty} Z_s^2 ds]$ in the case of strong convergence (resp. $\sup_{n \in \N}\ex[\int_{0}^{\infty} (Z^n_{\left\lceil ns\right\rceil})^2 ds]<\infty$ in the case of weak convergence).
\end{enumerate}
Moreover, in (ii), $\frac{1}{n}\sum_{i=1}^{\infty} (S^n Z^n_i)(\check{g}^n)\check{h}^n(i) $ can be replaced by $\frac{1}{n}\sum_{i=1}^{\infty} (S^n Z^n_i)(\check{g}^n\eins_{\N \setminus \{i\}})\check{h}^n(i)$.
\end{theorem}
\begin{proof}
 We wish to apply Lemma \ref{lem:weakconvergence} in order to prove the equivalence of (i) and (ii). As 
 $L^2(\Omega\times [0,\infty))=L^2(\Omega,\F,P)\otimes L^2([0,\infty))$ (with the tensor product in the sense of Hilbert spaces), the set 
 $\{\exp^{\diamond}(I(g)) h;\, g,h \in \mathcal{E} \}$ is total in $L^2(\Omega\times [0,\infty))$. In view of Proposition \ref{proposition_Wiener_and_Wick_exp} and
 (\ref{eq:hilf0001}), $(\exp^{\diamond_n}(I^n(\check g^n)) \check h^n(\lceil n \cdot \rceil))_{n\in \N}$ converges to $\exp^{\diamond}(I(g)) h$ 
 strongly in $L^2(\Omega\times [0,\infty))$ for every $g,h \in \mathcal{E}$. As 
 $$
 \frac{1}{n}\sum_{i=1}^{\infty} (S^n Z^n_i)(\check{g}^n)\check{h}^n(i)
 = \left\langle Z^n_{\lceil n \cdot \rceil}, e^{\diamond_n}(I^n(\check{g}^n))\check{h}^n(\lceil n \cdot \rceil)\right\rangle_{L^2(\Omega \times [0,\infty))},
 $$
 Lemma \ref{lem:weakconvergence} applies indeed.
 
  We finally note, that the `Moreover'-part of the assertion 
is an immediate consequence of the Cauchy-Schwarz inequality and the estimate
\begin{eqnarray*}
 &&\ex\left[\left( \exp^{\diamond_n}(I^n(\check{g}^n))-\exp^{\diamond_n}(I^n(\check{g}^n\eins_{\N \setminus \{i\}}))\right)^2 \right]
 \\ &=& 
 \ex\left[\left(\exp^{\diamond_n}(I^n(\check{g}^n\eins_{\N \setminus \{i\}}))\right)^2 \right]\ex\left[\left( \check{g}^n(i)\xi^{n}_{i}/\sqrt{n}\right)^2 \right]
 \\ &\leq&  \exp(\|\check{g}^n\|_{L^2_n(\N)}^2) \sup_{j\in\N} |g(j)|^2/n \rightarrow 0,
\end{eqnarray*}
making use of (\ref{eq:exponential_L2}) in the last line.
\end{proof}

We are now ready to give the proof of Theorem \ref{thm:Skorokhod_convergence_1}.

\begin{proof}[Proof of Theorem \ref{thm:Skorokhod_convergence_1}]
As the implication `$(ii)\Rightarrow (i)$' is trivial, we only have to show the converse implication. To this end, note first that,
by Proposition \ref{prop:Skorokhod_Stransform_discrete} and Theorem \ref{thm_S_transform_process_convergence_discrete},
 for every $g\in \mathcal{E}$,
 \begin{equation}\label{eq:hilf0009}
  \lim_{n\rightarrow \infty} (S^n\delta^n(Z^n))(\check{g}^n)=\int_0^\infty (SZ_t)(g)g(t)dt.
 \end{equation}
 As the sequence $(\delta^n(Z^n))_{n\in \N}$ is norm bounded by (i), it has a weakly convergent subsequence \cite[Theorem V.2.1]{Yosida}.
 We denote its limit by $X$.
 Then, 
applying Theorem \ref{thm_S_transform_convergence_1} and (\ref{eq:hilf0009}) along the subsequence, we obtain, for every $g\in \mathcal{E}$,
\begin{equation}\label{eq:hilf0023}
(SX)(g)=\int_0^\infty (SZ_t)(g)g(t)dt.
\end{equation}
Hence, by Definition \ref{def:Skorokhod}, $Z\in D(\delta)$ and $\delta(Z)=X$. Finally, by Theorem \ref{thm_S_transform_convergence_1}
and (\ref{eq:hilf0009})--(\ref{eq:hilf0023}),
weak $L^2(\Omega,\F,P)$-convergence of  $(\delta^n(Z^n))_{n\in \N}$ to $\delta(Z)$ holds along the whole sequence, and not  only along the subsequence.
\end{proof}

We now turn to the weak approximation of the Malliavin derivative.
Again, we apply a definition in terms of the $S$-transform, which we show to be equivalent to the more classical one in terms of the chaos
decomposition in the Appendix.
\begin{definition}\label{prop:Malliavin_domain}
 A random variable $X\in L^2(\Omega,\F,P)$ is said to belong to the domain $\D^{1,2}$ of the Malliavin derivative, if  there 
 is a stochastic process $Z\in L^2(\Omega\times [0,\infty))$ such that for every $g,h\in \mathcal{E}$,
 $$
 \int_0^\infty  (SZ_s)(g) h(s) ds =\ex\left[X \exp^{\diamond}(I(g)) \left(I(h) - \int_0^\infty  g(s)h(s) ds \right) \right].
 $$
 In this case,  $Z$ is unique and $DX:=Z$ is called the \emph{Malliavin derivative} $X$.
\end{definition}

For every $X\in L^2(\Omega,\F,P)$ we define the discretized Malliavin derivative of $X$ at $j\in \N$ with respect to $(\xi^n_i)_{i\in \N}$
by
$$
D^n_jX:= \sqrt{n} \ex[\xi^n_j X| \F^n_{-j}].
$$
We observe that, for fixed $j$, $D^n_j$ is a continuous linear operator from $L^2(\Omega,\F,P)$ to $L^2(\Omega,\F,P)$, because 
by H\"older's inequality for conditional expectations and the independence of the family $(\xi^n_i)_{i\in \N}$, 
$$
|D^n_jX|^2\leq n  \ex[X^2|\F^n_{-j}]\,\ex[(\xi^n_j)^2| \F^n_{-j}]=
n  \ex[X^2|\F^n_{-j}].
$$

We say that $X$ belongs to the domain
$\D^{1,2}_n$
of the discretized Malliavin derivative, if the process $D^nX:=(D^n_iX)_{i\in \N}$ is a member of $L^2_n(\Omega\times \N)$. In this case 
$D^nX$ is called the \emph{discretized Malliavin derivative} of $X$ with respect to $(\xi^n_i)_{i\in \N}$. 
As $D^n_j$ is continuous for fixed $j$, it is easy to check that the discretized Malliavin derivative
is a densely defined closed operator from $L^2(\Omega,\F,P)$ to $L^2_n(\Omega\times \N)$. 

In the following theorem
and in the remainder of the paper we use the convention $Z^n_0=0$ for $Z^n\in L^2_n(\Omega\times \N)$.

\begin{theorem}\label{thm_Malliavin_derivative_convergence_1}
 Suppose $(X^n)_{n\in \N}$ converges to $X$ weakly in $L^2(\Omega,\F,P)$ and  $X^n\in \D^{1,2}_n$ for every $n\in \N$. Then, the following 
 are equivalent:
 \begin{enumerate}
  \item[(i)] $\sup_{n\in \N} \frac{1}{n} \sum_{i=1}^\infty \ex[(D^n_iX^n)^2]<\infty$.
  \item[(ii)] $X\in \D^{1,2}$ and $(D^n_{\lceil n \cdot\rceil} X^n)_{n\in \N}$ converges to $DX$ weakly in $L^2(\Omega\times [0,\infty))$. 
 \end{enumerate}
\end{theorem}
 The proof is prepared by  two propositions. The first one contains the duality relation between the discrete Skorokhod integral and discretized 
Malliavin derivative.
\begin{proposition}\label{prop:duality_discrete}
For every $n\in \N$, the discrete Skorokhod integral is the adjoint operator of the discretized Malliavin derivative.
In particular, $\delta^n$ is closed and, for every
 $X\in \D^{1,2}_n$ and $Z^n\in D(\delta^n)$,
 $$
 \frac{1}{n} \sum_{i=1}^\infty  \ex\left[Z^n_i D^n_iX \right]=\ex[\delta^n(Z^n) X].
 $$
\end{proposition}
We emphasize that, choosing $X=\exp^{\diamond_n}(I^n(\check g^n))$, $g\in \mathcal{E}$, in Proposition \ref{prop:duality_discrete}, we obtain the assertion of Proposition \ref{prop:Skorokhod_Stransform_discrete}.
Indeed, we only have to note that,
for every $f^n\in L^2_n(\N)$,
 \begin{equation*}
 D^n_i \exp^{\diamond_n}(I^n(f^n))=f^n(i) \exp^{\diamond_n}(I^n(f^n\eins_{\N\setminus \{i\}})).
 \end{equation*}
\begin{proof}
Suppose first, that $Z^n\in D(\delta^n)$ and $X\in \D^{1,2}_n$. Then, for every $M\in \N$, and $i\in \N$, 
$$
\ex\left[\left|\sqrt{n} \ex[\xi^n_i X| \F^n_{M,-i}]\right|^2\right]\leq \ex\left[\left|D^n_iX\right|^2\right].
$$
Hence, by the martingale convergence theorem and dominated convergence,
$$
\lim_{M\rightarrow \infty} \frac{1}{n}\sum_{i=1}^\infty 
\ex\left[\left|\sqrt{n} \ex[\xi^n_i X|\F^n_{M,-i}]- D^n_iX\right|^2 \right] =0.
$$
Consequently,
 \begin{eqnarray*}
\frac{1}{n} \sum_{i=1}^\infty  \ex\left[Z^n_i D^n_iX \right]
 &=& \lim_{M\rightarrow \infty} \frac{1}{n} \sum_{i=1}^M  \ex\left[Z^n_i  \sqrt{n} \ex[\xi^n_i X|\F^n_{M,-i}] \right]
 \\ &=&
  \lim_{M\rightarrow \infty}\frac{1}{\sqrt n} \sum_{i=1}^M  \ex\left[ X \xi^n_i   \ex[Z^n_i|\F^n_{M,-i}] \right]
 \\ &=& \lim_{M\rightarrow \infty} \ex\left[X \sum_{i=1}^M  \ex[Z^n_i|\F^n_{M,-i}] \frac{\xi^n_i}{\sqrt n} \right] 
 =
  \ex[ X\delta^n(Z^n)]. 
 \end{eqnarray*}
Conversely, suppose that $Z^n$ is in the domain of the adjoint operator of the discretized Malliavin derivative, i.e., there 
is an $Y^n\in L^2(\Omega,\F,P)$ such that for every $X\in \D^{1,2}_n$,
\begin{equation}\label{eq:hilf_duality}
 \frac{1}{n} \sum_{i=1}^\infty  \ex\left[Z^n_i D^n_iX \right]=\ex[Y^n X].
\end{equation}
We first note that, by construction, $X\in \D^{1,2}_n$ if and only $\ex[X|\F^n] \in \D^{1,2}_n$, and, if this is the case, both random variables
have the same discretized Malliavin derivative. Hence, applying the duality relation (\ref{eq:hilf_duality}), with $X$ and $\ex[X|\F^n]$,
we observe that, $Y^n=\ex[Y^n|\F^n]$. Now suppose that $X \in L^2(\Omega,\F^n_M,P)$. Then $X\in \D^{1,2}_n$, $D^n_i X=\sqrt{n} \ex[\xi^n_i X|\F^n_{M,-i}]$  for every $i\leq M$, and $D^n_i X=0$ for $i>M$. Hence,
(\ref{eq:hilf_duality}) and the same manipulations as above imply
\begin{eqnarray*}
\ex[Y^n X]= \ex\left[X \sum_{i=1}^M  \ex[Z^n_i|\F^n_{M,-i}] \frac{\xi^n_i}{\sqrt n} \right],
\end{eqnarray*}
i.e. 
$$
\ex[Y^n|\mathcal{F}^n_M]=\sum_{i=1}^M  \ex[Z^n_i|\F^n_{M,-i}] \frac{\xi^n_i}{\sqrt n}.
$$
By the martingale convergence theorem, $(\ex[Y^n|\mathcal{F}^n_M])_{M\in \N}$ converges strongly in $L^2(\Omega,\F,P)$ to 
$\ex[Y^n|\F^n]=Y^n$. Hence, $Z^n \in D(\delta^n)$ and $\delta^n(Z^n)=Y^n$. Finally, closedness is a general property of 
adjoint operators, see, e.g., \cite[p. 196]{Yosida}.
\end{proof}

The next proposition is a consequence of the weak convergence result for discrete Skorokhod integrals in Theorem \ref{thm:Skorokhod_convergence_1}.

\begin{proposition}\label{prop:Skorokhod_simple}
 For every $g,h \in \mathcal{E}$, 
 $$
 \lim_{n\rightarrow \infty} \delta^n( \exp^{\diamond_n}(I^n(\check g^n))\, \check h^n) = \exp^{\diamond}(I(g)) \left(I(h) - \int_0^\infty  g(s)h(s) ds \right) 
 $$
 strongly in $L^2(\Omega,\F,P)$.
\end{proposition}
\begin{proof}
 Notice first that, for fixed $n\in \N$, $\exp^{\diamond_n}(I^n(\check g^n))\, \check h^n\in D(\delta^n)$, because $\check h^n(i)$ vanishes, 
 if $i$  is sufficiently large. A direct computation, making use of (\ref{eq:discrete_Skorokhod_finite}), shows
 $$
 \delta^n( \exp^{\diamond_n}(I^n(\check g^n))\, \check h^n) = \sum_{i=1}^\infty \exp^{\diamond_n}(I^n(\check g^n \eins_{\N\setminus \{i\}})) \check h^n(i) \frac{\xi^n_i}{\sqrt n}.
 $$
 For $i\neq j$ we obtain, by independence of $(\xi^n_k)_{k\in \N}$,
$$
 \ex\left[\exp^{\diamond_n}(I^n(\check g^n \eins_{\N\setminus \{i\}}))\exp^{\diamond_n}(I^n(\check g^n \eins_{\N\setminus \{j\}}))\xi^n_i \xi^n_j \right] 
 = \check g^n(i) \frac{1}{\sqrt{n}}\check g^n(j) \frac{1}{\sqrt{n}} \prod_{k\in \N\setminus \{i,j\}} \left(1+\check g^n(k)^2 \frac{1}{n}\right).
$$
 Combining this with an analogous calculation for the case $i=j$ yields
 \begin{eqnarray*}
  \ex\left[ \left|\delta^n( \exp^{\diamond_n}(I^n(\check g^n))\, \check h^n)\right|^2\right]&=& \frac{1}{n^2}\sum_{i,j=1,\; i\neq j}^\infty \check h^n(i)\check g^n(i) \check h^n(j)\check g^n(j) \prod_{k\in \N\setminus \{i,j\}} \left(1+\check g^n(k)^2 \frac{1}{n}\right)
  \\ &&+ \frac{1}{n} \sum_{i=1}^\infty  \check h^n(i)^2 \prod_{k\in \N\setminus \{i\}} \left(1+\check g^n(k)^2 \frac{1}{n}\right).
 \end{eqnarray*}
As $g$ and $h$ are bounded with compact support, it is straightforward to check in view of (\ref{eq:hilf0001}) that
\begin{equation}\label{eq:hilf0024}
 \lim_{n\rightarrow \infty}  \ex\left[ \left|\delta^n( \exp^{\diamond_n}(I^n(\check g^n))\, \check h^n)\right|^2\right]=
 e^{\int_0^\infty g(s)^2ds} \left( \left(\int_0^\infty h(s)g(s)ds\right)^2+ \int_0^\infty h(s)^2ds \right).
\end{equation}
Thus,  $(\delta^n(\exp^{\diamond_n}(I^n(\check g^n))\, \check h^n))_{n\in \N}$ 
converges  to $\delta(\exp^{\diamond}(I( g))\,  h)$ weakly in $L^2(\Omega,\F,P)$ by Theorem \ref{thm:Skorokhod_convergence_1}. 
The identity
$$
\delta(\exp^{\diamond}(I( g))\,  h)= \exp^{\diamond}(I(g)) \left(I(h) - \int_0^\infty  g(s)h(s) ds \right) 
$$
can either be derived by a direct computation making use of the $S$-transform definition of the Skorokhod integral (Definition 
\ref{def:Skorokhod}) or alternatively is a simple consequence of \cite[Proposition 1.3.3]{Nualart} in conjunction with Definition 1.2.1
in the same reference. Applying the Cameron-Martin shift \cite[Theorem 14.1]{Janson} twice, we observe
\begin{eqnarray*}
 && \ex\left[\left(\exp^{\diamond}(I(g)) \left(I(h) - \int_0^\infty  g(s)h(s) ds \right) \right)^2\right]=
  e^{\int_0^\infty g(s)^2ds} \ex\left[ \exp^{\diamond}(I(g)) I(h)^2 \right] \\ &=& 
  e^{\int_0^\infty g(s)^2ds} \ex\left[ \left(I(h)+  \int_0^\infty  g(s)h(s) ds \right)^2 \right]
  \\ &=& e^{\int_0^\infty g(s)^2ds} \left( \left(\int_0^\infty h(s)g(s)ds\right)^2+ \int_0^\infty h(s)^2ds \right).
\end{eqnarray*}
Thanks to (\ref{eq:hilf0024}), this turns weak into strong convergence.
\end{proof}

The proof of Theorem \ref{thm_Malliavin_derivative_convergence_1} is now analogous to that of Theorem \ref{thm:Skorokhod_convergence_1}.

\begin{proof}[Proof of Theorem \ref{thm_Malliavin_derivative_convergence_1}]
`$(ii)\Rightarrow (i)$' is obvious, since $\frac{1}{n} \sum_{i=1}^\infty \ex[(D^n_i X^n)^2]=\int_0^\infty \ex[(D^n_{\lceil ns\rceil} X^n)^2 ] ds  $.
`$(i)\Rightarrow (ii)$': Notice first that, for every $g,h \in \mathcal{E}$, by Proposition \ref{prop:duality_discrete} with $Z^n=\exp^{\diamond_n}(I^n(\check g^n))\, \check h^n$ and Proposition 
 \ref{prop:Skorokhod_simple},
 \begin{align}\label{eq:hilf0021}
  \lim_{n\rightarrow \infty} \frac{1}{n}\sum_{i=1}^{\infty} (S^n D^n_iX^n)(\check{g}^n)\check{h}^n(i) &
  = \lim_{n \rightarrow \infty} \ex[X^n\delta^n(\exp^{\diamond_n}(I^n(\check{g}^n))\check{h}^n)] \nonumber\\
  &=\ex\left[X \exp^{\diamond}(I(g)) \left(I(h) - \int_0^\infty  g(s)h(s) ds \right) \right],
 \end{align}
since $(X^n)$ converges to $X$ weakly in $L^2(\Omega,\F,P)$. As the sequence $(D^n_{\lceil n\cdot \rceil}X^n)_{n\in \N}$ is norm bounded 
in $L^2(\Omega\times [0,\infty))$ by (i), it has a weakly convergent subsequence. We denote its limit by $Z$. Applying (\ref{eq:hilf0021}) 
and Theorem \ref{thm_S_transform_process_convergence_discrete} along this subsequence, we conclude
\begin{equation}\label{eq:hilf0022}
\int_0^\infty (SZ_s)(g) h(s) ds= \ex\left[X \exp^{\diamond}(I(g)) \left(I(h) - \int_0^\infty  g(s)h(s) ds \right) \right].
\end{equation}
Hence, $X\in \D^{1,2}$ and $DX=Z$ by Definition \ref{prop:Malliavin_domain}. Finally, applying (\ref{eq:hilf0021})--(\ref{eq:hilf0022}) and Theorem \ref{thm_S_transform_process_convergence_discrete}
along the whole sequence $(D^n_{\lceil n\cdot \rceil}X^n)_{n\in \N}$, shows that this sequence converges weakly in $L^2(\Omega\times [0,\infty))$
to $DX$.
\end{proof}

In order to check the assumptions of Theorem \ref{thm:Skorokhod_convergence_1}, 
we consider the space $\L^{1,2}_n$, which consists of processes $Z^n\in L^2_n(\Omega\times \N)$ such that $Z^n_i \in \D^{1,2}_n$ for every
$i\in \N$ and
\begin{equation}\label{eq:L12_n}
\frac{1}{n^2} \sum_{i,j=1,\;i\neq j}^\infty \ex\left[ |D^n_jZ^n_i|^2 \right]<\infty.
\end{equation}
\begin{proposition}\label{prop:Skorokhod_L2}
 For every $n\in \N$, $\L^{1,2}_n\subset D(\delta^n)$ and, for $Z^n\in \L^{1,2}_n$,
 \begin{eqnarray}
\delta^n(Z^n)&=& \sum_{i=1}^\infty \ex[Z^n_i|\F^n_{-i}] \frac{\xi^n_i}{\sqrt n},\quad \textnormal{(strong }L^2(\Omega,\F,P)\textnormal{-convergence)},  \label{eq:L12_representation}\\
 \ex\left[\left(\delta^n(Z^n)\right)^2 \right]&=& \frac{1}{n} \sum_{i=1}^\infty \ex\left[ \ex\left[Z^n_i|\F^n_{-i} \right]^2 \right]
 +   \frac{1}{n^2} \sum_{i,j=1, \; i\neq j}^\infty \ex\left[(D^n_i Z_j^n) (D^n_j Z_i^n) \right]. \label{eq:L12_discrete}
 \end{eqnarray}
 In particular, in the context of Theorem \ref{thm:Skorokhod_convergence_1}, assertion (i) is equivalent to
 \begin{enumerate}
  \item[(i')] $\sup_{n\in \N} \frac{1}{n^2} \sum_{i,j=1, \; i\neq j}^\infty \ex\left[(D^n_i Z_j^n) (D^n_j Z_i^n) \right]<\infty$,
 \end{enumerate}
 if we additionally assume that $Z^n\in \L^{1,2}_n$ for every $n\in \N$.
\end{proposition}
\begin{proof}
 Fix $N_1<N_2\in \N$. Then,
 \begin{eqnarray*}
  &&\ex\left[\left(  \sum_{i=N_1}^{N_2} \ex[Z^n_i|\F^n_{-i}] \frac{\xi^n_i}{\sqrt{n}}\right)^2 \right]= \frac{1}{n}\sum_{i=N_1}^{N_2}  
  \ex\left[ \ex[Z^n_i|\F^n_{-i}]^2 \left(\xi^n_i\right)^2\right]\\ & +&\frac{1}{n} \sum_{i,j=N_1, \; i\neq j}^{N_2} 
  \ex \left[ \ex[Z^n_i|\F^n_{-i}]\ex[Z^n_j|\F^n_{-j}] \xi^n_i \xi^n_j \right]=(I)_{N_1,N_2}+(II)_{N_1,N_2}.
 \end{eqnarray*}
By the independence of the discrete-time noise $(\xi^n_i)_{i\in \N}$ and as the conditional expectation has norm 1,
\begin{align}
(I)_{1,N}=\frac{1}{n}\sum_{i=1}^{N}  
  \ex\left[ \ex[Z^n_i|\F^n_{-i}]^2 \right]\rightarrow \frac{1}{n}\sum_{i=1}^\infty 
  \ex\left[ \ex[Z^n_i|\F^n_{-i}]^2 \right]<\infty,\quad N\rightarrow \infty,\label{eq:I_Mconv} 
\end{align}
and $(I)_{N_1,N_2}\rightarrow 0$  as $N_1,N_2$ tend to infinity. 
In order to treat $(II)_{N_1,N_2}$, we first note 
that for  
any random variable  $X^n\in L^1(\Omega,\F^n,P)$ and $i\neq j\in \N$, by Fubini's theorem,
\begin{equation}\label{eq:-i,-j}
\ex\left[\left. \ex\left[ X^n|\F^n_{-i} \right]\right| \F^n_{-j}\right]=
\ex\left[\left. \ex\left[ X^n|\F^n_{-j} \right]\right| \F^n_{-i}\right].
\end{equation}
Hence, for $i\neq j\in \N$,
\begin{eqnarray*}
 && \ex \left[ \ex[Z^n_i|\F^n_{-i}]\ex[Z^n_j|\F^n_{-j}] \xi^n_i \xi^n_j \right]=
 \ex \left[ \ex[Z^n_i\xi^n_j|\F^n_{-i}]\ex[Z^n_j\xi^n_i|\F^n_{-j}]  \right] \nonumber\\
 &=&   \ex \left[ \ex\left[ \left.\ex[Z^n_i\xi^n_j|\F^n_{-i}]\right|\F^n_{-j}\right] Z^n_j\xi^n_i  \right]
 =   \ex \left[ \ex\left[ \left.\ex[Z^n_i\xi^n_j|\F^n_{-j}]\right|\F^n_{-i}\right] Z^n_j\xi^n_i  \right] \nonumber
 \\ &=&  \ex \left[ \ex[Z^n_i\xi^n_j|\F^n_{-j}]\ex[Z^n_j\xi^n_i|\F^n_{-i}]  \right]=\frac{1}{n} \ex\left[(D^n_iZ^n_j)(D^n_jZ^n_i) \right]. \label{eq:II_Mfubini}
\end{eqnarray*}
Consequently, by Young's inequality,
$$
 n \left|\ex \left[ \ex[Z^n_i|\F^n_{-i}]\ex[Z^n_j|\F^n_{-j}] \xi^n_i \xi^n_j \right]\right|\leq 
 \frac{1}{2} \ex \left[ (D^n_i Z^n_j)^2\right] + \frac{1}{2} \ex \left[ (D^n_j Z^n_i)^2\right]. 
$$
The $\L^{1,2}_n$-assumption, thus, ensures that
\begin{eqnarray*}
\lim_{N\rightarrow \infty}  (II)_{1,N} &=&  \frac{1}{n^2} \sum_{i,j=1, i \neq j}^\infty \ex\left[(D^n_i Z^n_j)(D^n_j Z^n_i) \right] < \infty
\end{eqnarray*}
and  $(II)_{N_1,N_2}\rightarrow 0$  as $N_1,N_2$ tend to infinity. Hence, by (\ref{eq:discrete_Skorokhod_finite}), the sequence $(\delta^n(Z^n\eins_{[1,N]}))_{N \in \N}$ 
is Cauchy in $L^2(\Omega,\F,P)$. By the closedness of the discrete Skorokhod integral, $Z^n\in D(\delta^n)$ and  we obtain $\L^{1,2}_n \subset D(\delta^n)$, (\ref{eq:L12_representation}) 
and (\ref{eq:L12_discrete}).
We finally suppose that the assumptions of Theorem \ref{thm:Skorokhod_convergence_1} are in force and that 
 $Z^n\in \L^{1,2}_n$ for every $n\in \N$. Then,
 $$
 \sup_{n\in \N}\frac{1}{n}\sum_{i=1}^\infty 
  \ex\left[ \ex[Z^n_i|\F^n_{-i}]^2 \right]<\infty,
 $$
 because of the assumed weak convergence of the sequence  $(Z^n_{\lceil n\cdot\rceil})_{n\in\N}$. Thus, the sequence 
 $(\delta^n(Z^n))_{n\in \N}$ is norm bounded in $L^2(\Omega,\F,P)$, if and only if (i') holds.
\end{proof}

As a consequence of the previous proposition, we obtain the following strong $L^2(\Omega,\F,P)$-convergence results 
to the Malliavin derivative.

\begin{theorem}\label{thm:D22}
 Suppose $(X^n)_{n\in\N}$ converges to $X$ strongly in $L^2(\Omega,\F,P)$. Moreover assume that
 $X^n\in \D^{2,2}_n$ for every $n\in \N$, i.e. 
 $$
 \frac{1}{n} \sum_{i=1}^\infty \ex\left[(D^n_iX)^2\right] +  \frac{1}{n^2} \sum_{i,j=1,\,i\neq j}^\infty \ex\left[(D^n_jD^n_iX)^2\right]<\infty.
 $$
 Then, the following assertions are equivalent:
\begin{itemize}
 \item[(i)]   $ \sup_{n\in \N}\left(\frac{1}{n} \sum_{i=1}^\infty \ex\left[(D^n_iX)^2\right] +  \frac{1}{n^2} \sum_{i,j=1,\,i\neq j}^\infty \ex\left[(D^n_jD^n_iX)^2\right]\right)<\infty$ .
 \item[(ii)]  $X\in \D^{1,2}$, $DX \in D(\delta)$, $(D^n_{\lceil n\cdot \rceil}X^n)_{n\in \N}$ converges to $DX$ strongly in 
 $L^2(\Omega\times[0,\infty))$, and $(\delta^n(D^nX^n))_{n\in \N}$ converges to $\delta(DX)$ weakly in $L^2(\Omega,\F,P)$.
 \end{itemize}
\end{theorem}

\begin{remark}
 Recall that $L=-\delta\circ D$ is the infinitesimal generator of the Ornstein-Uhlenbeck semigroup, see \cite[Section 1.4]{Nualart}, and is sometimes called 
Ornstein-Uhlenbeck operator (cf. also \cite[Example 4.7]{Janson}). So the previous theorem provides, at the same time, 
sufficient conditions for the strong convergence to the Malliavin derivative and the weak convergence to the Ornstein-Uhlenbeck operator.
\end{remark}

\begin{proof}
 Let $Z^n_i=D^n_i X^n$. Then, $X^n\in \D^{2,2}_n$ implies $Z^n\in \L^{1,2}_n$. Note that, for $i\neq j$, by (\ref{eq:-i,-j}),
 $$
 D^n_j Z^n_i= D^n_j D^n_i X= D^n_i D^n_j X= D^n_i Z^n_j, 
 $$
 i.e. $(D^n_jZ^n_i)(D^n_i Z^n_j)=  (D^n_j D^n_i X)^2$. Hence, by Theorem \ref{thm_Malliavin_derivative_convergence_1} and 
 Theorem \ref{thm:Skorokhod_convergence_1} in conjunction with Proposition \ref{prop:Skorokhod_L2}, assertion (i) is equivalent to
 \begin{itemize}
 \item[(ii')]  $X\in \D^{1,2}$, $DX \in D(\delta)$, $(D^n_{\lceil n\cdot \rceil}X^n)_{n\in \N}$ converges to $DX$ weakly in 
 $L^2(\Omega\times[0,\infty))$, and $(\delta^n(D^nX^n))_{n\in \N}$ converges to $\delta(DX)$ weakly in $L^2(\Omega,\F,P)$.
 \end{itemize}
 So we only need to show that under (ii') the convergence of $(D^n_{\lceil n\cdot \rceil}X^n)_{n\in \N}$ to $DX$ holds true in the strong topology.
 However, by the duality relation in Proposition \ref{prop:duality_discrete}, the weak $L^2(\Omega,\F,P)$-convergence 
 of $(\delta^n(D^nX^n))_{n\in \N}$ and the strong $L^2(\Omega,\F,P)$-convergence of $(X^n)_{n\in \N}$,
 $$
  \int_0^\infty \ex[(D^n_{\lceil n\cdot \rceil}X^n)^2]dt=\ex[\delta^n(D^nX^n) X^n]\rightarrow \ex[\delta(DX) X ]= \int_0^\infty \ex[(D_tX)^2]dt,  
 $$
 making use of the continuous time duality  between Skorokhod integral and Malliavin derivative in the last step.
\end{proof}

The analogous result for the Skorokhod integral reads as follows.

\begin{theorem}\label{thm:L22}
Suppose $(Z^n_{\lceil n\cdot\rceil})_{n\in\N}$ converges strongly to $Z$ in $L^2(\Omega\times [0,\infty))$ and  assume that $Z^n \in \L^{2,2}_n$, i.e., for every $n\in \N$,
$$
\frac{1}{n^2} \sum_{i,j=1,\;i\neq j}^\infty \ex\left[ |D^n_iZ^n_j|^2 \right]+\frac{1}{n^3} \sum_{i,j,k=1,\; |\{i,j,k\}|=3}^\infty \ex\left[ |D^n_i D^n_j Z^n_k|^2 \right]<\infty.
$$
Then the following assertions are equivalent:
\begin{itemize}
 \item[(i)]  $\sup_{n\in \N} \left(\frac{1}{n^2} \sum_{i,j=1, \; i\neq j}^\infty \ex\left[(D^n_i Z_j^n) (D^n_j Z_i^n) \right]\right)<\infty$ and 
$$
\sup_{n\in \N} \left(\frac{1}{n^2} \sum_{i,j=1, \; i\neq j}^\infty \ex\left[\ex[D^n_i Z_j^n|\F^n_{-j}]^2 \right]+ 
\frac{1}{n^3} \sum_{i,j,k=1,\; |\{i,j,k\}|=3} \ex\left[ (D^n_i D^n_j Z^n_k) (D^n_k D^n_j Z^n_i) \right]  \right)<\infty.
$$
 \item[(ii)] $Z \in D(\delta)$, $\delta(Z) \in \D^{1,2}$, $(\delta^n(Z^n))_{n \in \N}$ converges to $\delta(Z)$ strongly in $L^2(\Omega,\F,P)$ 
and $(D^n_{\lceil n \cdot \rceil} \delta^n(Z^n))_{n \in \N}$ converges  to $D \delta (Z)$ weakly in $L^2(\Omega \times [0,\infty))$.
\end{itemize}
\end{theorem}

As a preparation we explain how to compute the dicretized Malliavin derivative of a discrete Skorokhod integral, which is analogous 
to the continuous-time situation, cp. e.g. \cite[Proposition 1.3.8]{Nualart}.

\begin{proposition}\label{prop:Malliavin_Skorokhod}
Suppose $Z^n \in \L^{1,2}_n$. Then $(D^n_i Z^n) \eins_{\N\setminus\{i\}} \in D(\delta^n)$ for every $i\in \N$, and
$$
D^n_i \delta^n(Z^n)= \ex[Z^n_i|\F^n_{-i}]+  \delta^n(D^n_i Z^n \eins_{\N\setminus\{i\}}).
$$
\end{proposition}
\begin{proof}
 By (\ref{eq:L12_representation}) and the continuity of $D^n_i$,
$$
D^n_i \delta^n(Z^n)= D^n_i \left(\ex[Z^n_i|\F^n_{-i}] \frac{\xi^n_i}{\sqrt n}\right) +\sum_{j=1,\; j\neq i}^\infty D^n_i \left(\ex[Z^n_j|\F^n_{-j}] \frac{\xi^n_j}{\sqrt n}\right),
$$
(including the strong convergence of the series on the right-hand side in $L^2(\Omega,\F,P)$).
By (\ref{eq:-i,-j}), for $i \neq j$,
$$
\ex[\xi^n_i \ex[Z^n_j|\F^n_{-j}] \xi^n_j|\F^n_{-i}] = \ex[\ex[\xi^n_i Z^n_j|\F^n_{-j}] |\F^n_{-i}] \xi^n_j = \ex[\ex[\xi^n_i Z^n_j|\F^n_{-i}] |\F^n_{-j}] \xi^n_j.
$$
Moreover,
$$
\ex[(\xi^n_i)^2 \ex[Z^n_i|\F^n_{-i}]|\F^n_{-i}]=  \ex[Z^n_i|\F^n_{-i}].
$$
Hence,
$$
D^n_i \delta^n(Z^n)= \ex[Z^n_i|\F^n_{-i}] +\sum_{j=1,\; j\neq i}^\infty \ex[D^n_i Z^n_j|\F^n_{-j}] \frac{\xi^n_j}{\sqrt n},
$$
and the closedness of the discrete Skorokhod integral concludes.
\end{proof}

\begin{proof}[Proof of Theorem \ref{thm:L22}] 
The $\L^{2,2}_n$-assumption guarantees that, for every $i\in \N$, $(D^n_i Z^n) \eins_{\N\setminus\{i\}} \in \L^{1,2}_n$. As $(Z^n_{\lceil n \cdot \rceil})_{n\in \N}$ is norm bounded
 in $L^2(\Omega\times [0,\infty))$ by the assumed 
strong convergence to $Z$, we observe in view of Propositions \ref{prop:Skorokhod_L2} and \ref{prop:Malliavin_Skorokhod} that (i) is equivalent to
\begin{itemize}
 \item[(i')] $\sup_{n\in \N} \ex[|\delta^n(Z^n)|^2]<\infty$ and $\sup_{n\in \N} \frac{1}{n} \sum_{i=1}^\infty \ex[|D^n_i\delta^n(Z^n)|^2]<\infty$.
 \end{itemize}
Thanks to Theorems  \ref{thm:Skorokhod_convergence_1} and \ref{thm_Malliavin_derivative_convergence_1}, assertion (i') is equivalent to
\begin{itemize}
 \item[(ii')] $Z\in D(\delta)$, $\delta(Z) \in \D^{1,2}$, $(\delta^n(Z^n))_{n\in \N}$ converges weakly to $\delta(Z)$ in 
 $L^2(\Omega,\mathcal{F},P)$, and $(D^n\delta^n(Z^n))_{n\in \N}$ converges to $D \delta(Z)$ weakly in $L^2(\Omega \times [0,\infty))$.
 \end{itemize}
Due to the strong convergence of $(Z^n_{\lceil n \cdot \rceil})_{n \in \N}$ to $Z$ and the weak convergence of $(D^n_{\lceil n \cdot \rceil} \delta^n(Z^n))_{n \in \N}$ to 
$D(\delta(Z))$, 
the continuous time duality between Skorokhod integral and Malliavin derivative and its discrete  time counterpart in Proposition \ref{prop:duality_discrete} imply
$$
\|\delta^n(Z^n)\|^2_{L^2(\Omega,\mathcal{F},P)} = \int_{0}^{\infty}\ex[Z^n_{\lceil n s \rceil} D^n_{\lceil n s \rceil} \delta^n(Z^n) ]ds \rightarrow \int_{0}^{\infty}\ex[Z_s D_s \delta(Z) ] ds = \|\delta(Z)\|^2_{L^2(\Omega,\mathcal{F},P)}.
$$
Hence we obtain the convergence of $(\delta^n(Z^n))_{n\in \N}$ to $\delta(Z)$ in the strong topology, i.e., assertion (ii') is equivalent to assertion (ii).
\end{proof}

\section{Strong and weak $L^2$-approximation of the It\^o integral and the Clark-Ocone derivative}
 \label{section:Ito}

In this section, we first specialize the approximation result for the Skorokhod integral to predictable integrands. In this way,
we obtain necessary and sufficient conditions for strong and weak  $L^2$-convergence of discrete It\^o integrals with respect to the noise
$(\xi^n_i)_{i \in \N}$ to It\^o integrals with respect to the Brownian motion $B$. Then, we discuss strong and weak $L^2$-approximations 
to the Clark-Ocone derivative, which provides the predictable integral representation of a random variable in $L^2(\Omega,\F,P)$
with respect to the Brownian motion $B$.

Suppose $Z^n \in L^2_{n}(\Omega \times \N)$ is predictable with respect to $(\F^n_i)_{i\in \N}$, i.e., for every $i\in \N$, 
$Z^n_i$ is measurable with respect to
$\F^n_{i-1}=\sigma(\xi^n_1,\ldots,\xi^n_{i-1})$. Then,
$$
\delta^n(Z^n)=\sum_{i=1}^\infty Z^n_i \frac{\xi^n_i}{\sqrt n}=:\int Z^n dB^n,
$$
which means that the discrete Skorokhod integral reduces to the discrete It\^o integral. Analogously, the Skorokhod integral $\delta(Z)$
is well-known to coincide with the It\^o integral $\int_0^\infty Z_s dB_s$, when $Z \in L^2(\Omega\times [0,\infty))$ is predictable
with respect to the augmented Brownian filtration $(\F_t)_{t\in [0,\infty)}$, see, e.g. \cite[Theorem 7.41]{Janson}.

In this case of predictable integrands, the approximation theorem for Skorokhod integrals (Theorem \ref{thm:Skorokhod_convergence_1})
can be improved as follows.

\begin{theorem}\label{thm:Ito_integral}
 Suppose $Z\in L^2(\Omega\times [0,\infty))$ is predictable with respect to the augmented Brownian filtration $(\F_t)_{t\in [0,\infty)}$,
 and, for every $n \in \N$, $Z^n\in L^2_{n}(\Omega \times \N)$ is predictable with respect to $(\F^n_i)_{i\in \N}$. 
 Then, the following are equivalent:
 \begin{enumerate}
  \item[(i)] $(Z^n_{\lceil n\cdot\rceil})_{n\in\N}$ converges to $Z$ strongly (resp. weakly) in $L^2(\Omega\times [0,\infty))$.
  \item[(ii)] The sequence of discrete It\^o integrals 
  $
  \left(  \int Z^n dB^n \right)_{n\in \N} 
  $
  converges strongly (resp. weakly) in $L^2(\Omega,\F,P)$ to $\int_0^\infty Z_s dB_s$. 
 \end{enumerate}
\end{theorem}

\begin{remark}
 We note that, in order to study convergence of It\^o integrals (with respect to different filtrations),  techniques of convergence 
 in distribution on the Skorokhod space of
 right-continuous functions with left limits are classically applied. E.g., the results by \cite{KP} immediately imply the following result in our setting: 
Suppose that $Z$ is predictable with respect to the Brownian filtration and its paths 
are right-continuous with left limits. Moreover, assume that $Z^n$ is predictable with respect to $(\F^n_i)_{i\in \N}$ and $(Z^n_{\lfloor 1+ n (\cdot)\rfloor})$ converges to $Z$ uniformly on compacts in probability.
Then, 
$$
\lim_{n\rightarrow \infty} \sum_{i=1}^{\lfloor n \cdot\rfloor} Z^n_i \frac{1}{\sqrt n}\xi^n_i=\int_0^\cdot Z_{s-} dB_s, 
$$
uniformly on compacts in probability. In contrast, our Theorem \ref{thm:Ito_integral} provides an $L^2$-theory and, in particular, includes the converse implication, namely 
that convergence of the discrete It\^o integrals implies convergence of the integrands.
\end{remark}

The proof of Theorem \ref{thm:Ito_integral} will make use of the following proposition.
\begin{proposition}\label{prop:Ito}
 Suppose $g,h\in \mathcal{E}$. Then, strongly in $L^2(\Omega\times [0,\infty))$,
 $$
 \lim_{n\rightarrow \infty} \exp^{\diamond_n}(I^n(\check g^n \eins_{[1,\lceil n \cdot\rceil-1]})) \check h^n(\lceil n \cdot\rceil)=\exp^{\diamond}(I( g \eins_{(0,\cdot]})) h(\cdot).
 $$
\end{proposition}
\begin{proof}
  Recall that the support of 
 $h$ is contained in 
 $[0,M]$ for some $M\in \N$. Hence, we can decompose,
 \begin{eqnarray*}
  &&\int_0^\infty \ex\left[\left(\exp^{\diamond_n}(I^n(\check g^n \eins_{[1,\lceil n t\rceil-1]})) {\check h^n}(\lceil nt \rceil)
  -\exp^{\diamond}(I(g \eins_{(0,t]})) h(t) \right)^2 \right]dt
  \\ &\leq& 
  2 \int_0^M \ex\left[\left(\exp^{\diamond_n}(I^n(\check g^n \eins_{[1,\lfloor n t\rfloor]})) 
  -\exp^{\diamond}(I( g \eins_{(0,t]}))  \right)^2 \right]h(t)^2  dt \\
  \\ && + 2 \int_0^\infty \ex\left[\left(\exp^{\diamond_n}(I^n(\check g^n \eins_{[1,\lceil n t\rceil-1]})) 
   \right)^2 \right] |{\check h^n}(\lceil nt \rceil)-h(t)|^2 dt,
 \end{eqnarray*}
 since $\lceil n t\rceil-1=\lfloor n t\rfloor$ for Lebesgue almost every $t\geq 0$.
As, by (\ref{eq:exponential_L2}) 
\begin{equation}\label{eq:hilfp1}
\sup_{n\in \N,\;t\in [0,\infty)}  \ex\left[\left(\exp^{\diamond_n}(I^n(\check g^n \eins_{[1,\lceil n t\rceil-1]})) \right)^2 \right] 
\leq  \sup_{n\in \N} \exp( \| {\check g^n}(\lceil n\cdot \rceil)\|^2_{L^2([0,\infty))}      ) <\infty, 
\end{equation}
the second term goes to zero by (\ref{eq:hilf0001}). Moreover, by the boundedness of $h$, the first one tends to zero by the dominated convergence theorem,
since, for every $t\in[0,\infty)$, by Proposition \ref{proposition_Wiener_and_Wick_exp},
$$
 \lim_{n\rightarrow \infty} \ex\left[\left(\exp^{\diamond_n}(I^n((\check{g}^n\eins_{[1,\lfloor n t\rfloor]}))) 
  -\exp^{\diamond}(I( g \eins_{(0,t]})) \right)^2 \right] =0.
$$
\end{proof}

\begin{proof}[Proof of Theorem \ref{thm:Ito_integral}]
 `$(i) \Rightarrow (ii)$': By the isometry for discrete It\^o integrals, we have
 \begin{equation}\label{eq:hilf0019}
\ex\left[\left(\int Z^n dB^n\right)^2  \right]= \ex\left[ \left|\sum_{i=1}^\infty  Z^n_i \frac{1}{\sqrt{n}}\xi_i^n\right|^2 \right]= \int_0^\infty \ex\left[ |Z^n_{\lceil ns\rceil}|^2 \right] ds. 
 \end{equation}
 Hence, if $(Z^n_{\lceil n\cdot \rceil})_{n\in \N}$ converges to $Z$ weakly in $L^2(\Omega,\F,P)$, then the left-hand side in (\ref{eq:hilf0019})
 is bounded in $n\in \N$, and so Theorem \ref{thm:Skorokhod_convergence_1} implies the asserted weak $L^2(\Omega,\F,P)$ convergence of  
 the sequence of discrete It\^o integrals to $\int_0^\infty Z_s dB_s$. If  $(Z^n_{\lceil n\cdot \rceil})_{n\in \N}$ converges to $Z$ strongly in $L^2(\Omega,\F,P)$,
 then, by (\ref{eq:hilf0019}) and the continuous time It\^o isometry,
 $$
 \lim_{n\rightarrow \infty} \ex\left[\left(\int Z^n dB^n\right)^2  \right]=\int_0^\infty \ex\left[ |Z_s|^2 \right]ds
 = \ex\left[\left(\int_0^\infty Z_s dB_s\right)^2 \right],
 $$
 which turns the weak $L^2(\Omega,\F,P)$-convergence of the sequence of discrete It\^o integrals into strong $L^2(\Omega,\F,P)$-convergence.
 \\[0.1cm]
 `$(ii) \Rightarrow (i)$': We first assume that the sequence of discrete It\^o integrals converges weakly in $L^2(\Omega,\F,P)$ 
 to the continuous time It\^o integral.
 By the implication `$(i) \Rightarrow (ii)$' (which we have already proved) and Proposition \ref{prop:Ito}, we obtain, for every $g,h\in \mathcal{E}$,
 \begin{equation}\label{eq:hilf0020}
  \lim_{n\rightarrow \infty} \sum_{i=1}^\infty  \exp^{\diamond_n}(I^n(\check g^n \eins_{[1,i-1]})) \check{h}^n(i)\frac{1}{\sqrt{n}}\xi_i^n=\int_0^\infty  \exp^\diamond(I(g \eins_{(0,s]}))\,h(s)\,dB_s
  \end{equation}
strongly in $L^2(\Omega,\F,P)$. As $Z^n$ is predictable and 
$$
\ex[ \exp^{\diamond_n}(I^n(\check g^n ))|\mathcal{F}^n_{i-1}]=\exp^{\diamond_n}(I^n(\check g^n \eins_{[1,i-1]})),
$$
we get, for every $g,h\in \mathcal{E}$, by the discrete It\^o isometry,
\begin{eqnarray*}
\frac{1}{n}\sum_{i=1}^{\infty} (S^n Z^n_i)(\check{g}^n)\check{h}^n(i) &=& 
\frac{1}{n}\sum_{i=1}^{\infty}\ex[\ex[Z^n_i|\mathcal{F}^n_{i-1}]\exp^{\diamond_n}(I^n(\check g^n)) \check{h}^n(i)] \\
&=&\frac{1}{n}\sum_{i=1}^{\infty} \ex[ Z^n_i \exp^{\diamond_n}(I^n(\check g^n \eins_{[1,i-1]})) \check{h}^n(i)] \\
&=& \ex\left[\left(\sum_{i=1}^\infty  Z^n_i \frac{1}{\sqrt{n}}\xi_i^n \right) \left( \sum_{i=1}^\infty  \exp^{\diamond_n}(I^n(\check g^n \eins_{[1,i-1]})) \check{h}^n(i)\frac{1}{\sqrt{n}}\xi_i^n\right) \right] .
\end{eqnarray*}
The assumed weak $L^2(\Omega,\F,P)$-convergence of the sequence of discrete It\^o integrals and the strong $L^2(\Omega,\F,P)$-convergence
in (\ref{eq:hilf0020}) now imply
$$
\lim_{n\rightarrow \infty} \frac{1}{n}\sum_{i=1}^{\infty} (S^n Z^n_i)(\check{g}^n)\check{h}^n(i)=
\ex\left[\left(\int_0^\infty  Z_s dB_s\right) \left( \int_0^\infty  \exp^\diamond(I(g \eins_{(0,s]}))\,h(s)\,dB_s\right) \right] .
$$
As $(\exp^\diamond(I(g \eins_{(0,s]})))_{s\in [0,\infty)}$ is a uniformly integrable martingale and $Z$ is predictable, we obtain,
by the It\^o isometry and the definition of the $S$-transform,
$$
\lim_{n\rightarrow \infty} \frac{1}{n}\sum_{i=1}^{\infty} (S^n Z^n_i)(\check{g}^n)\check{h}^n(i)=\int_{0}^{\infty}(S Z_s)(g) h(s) ds,\quad g,h \in \mathcal{E}.
$$
We can now apply Theorem \ref{thm_S_transform_process_convergence_discrete}. 
As $\int_0^\infty \ex\left[ |Z^n_{\lceil ns\rceil}|^2 \right] ds$ 
is bounded in $n$ by (\ref{eq:hilf0019}) and by the assumed weak $L^2(\Omega,\F,P)$-convergence of the discrete It\^o integrals, the latter 
Theorem implies that $(Z^n_{\lceil n\cdot \rceil})_{n\in \N}$ converges to $Z$ weakly in $L^2(\Omega\times [0,\infty))$. If we instead 
assume strong $L^2(\Omega,\F,P)$-convergence of the sequence of the discrete It\^o integrals, a straightforward application 
of the isometries for discrete and continuous-time It\^o integrals turns the weak $L^2(\Omega\times [0,\infty))$-convergence again 
into strong convergence. 
\end{proof}

We now turn to the Clark-Ocone derivative. Recall that a Brownian motion has the predictable representation property with respect to its 
natural filtration, i.e., for every $X\in L^2(\Omega,\F,P)$ there is a unique $(\F_t)_{t\in [0,\infty)}$-predictable process $\nabla X \in L^2(\Omega\times [0,\infty))$
such that
\begin{equation}\label{eq:predictable_representation}
 X=\ex[X]+\int_0^\infty \nabla_sX dB_s.
\end{equation}
We refer to $\nabla X$ as the \emph{generalized Clark-Ocone derivative} and recall that
$(\nabla_t X)_{t\geq 0}$ is the predictable projection of the Malliavin derivative $(D_tX)_{t\geq 0}$, if  $X\in \D^{1,2}$. By It\^o's isometry the operator $\nabla: L^2(\Omega,\F,P)\rightarrow L^2(\Omega\times [0,\infty))$
is continuous with  norm 1. 

Except in the case of binary noise, the discrete time approximation $B^{(n)}$ of the Brownian motion $B$ does not satisfy the discrete time
predictable representation property with respect to $(\F^n_i)_{i\in \N}$. Nonetheless one can consider the discrete time predictable 
projection of the discretized Malliavin derivative
$$
\nabla^n_i X:=\ex[ D^n_iX|\F^n_{i-1}]=\sqrt{n} \ex[\xi^n_i X| \F^n_{i-1}],\quad X\in L^2(\Omega,\F,P),\; i\in \N,
$$
as discretization of the generalized Clark-Ocone derivative. We refer 
to $(\nabla^n_i X)_{i\in \N}$ as \emph{discretized Clark-Ocone derivative of $X$} and note that it has been extensively studied 
in the context of discretization of backward stochastic differential equations, see, e.g., \cite{Briand_Delyon_Memin, Zhang,
GG}.

The operator
$$
\nabla^n: L^2(\Omega,\F,P)\rightarrow L^2_n(\Omega\times \N), \quad X \mapsto (\nabla^n_iX)_{i\in \N}
$$
is continuous with norm one. 
 Indeed, introducing the shorthand notation $\ex_{n,i}[\cdot]=\ex[\cdot|\F^n_i]$
and noting that the martingale $(\ex_{n,i}[X])_{i\in \N}$ is, for fixed $n\in \N$, uniformly integrable, and, thus, converges 
almost surely to $\ex[X|\F^n]$, as $i$ tends to infinity, one gets, by H\"older's and Jensen's inequality,
\begin{eqnarray*}
&& \frac{1}{n} \sum_{i=1}^\infty \ex\left[ \left(\sqrt{n}\, \ex_{n,i-1}\left[{\xi^n_i} X\right]\right)^2 \right]  
 =  \sum_{i=1}^\infty  \ex\left[ \left( \ex_{n,i-1}\left[{\xi^n_i} \left(\ex_{n,i}[X] -\ex_{n,i-1}[X]\right)\right]\right)^2 \right]\\
 &\leq & \sum_{i=1}^\infty  \ex\left[  \ex_{n,i-1}\left[({\xi^n_i})^2 \right]\ex_{n,i-1}\left[ \left(\ex_{n,i}[X] -\ex_{n,i-1}[X]\right)^2 \right]\right] \\
 &=& \ex\left[ \sum_{i=1}^\infty \left(\left(\ex_{n,i}[X]\right)^2 -\left(\ex_{n,i-1}[X]\right)^2\right) \right]
 =\ex\left[ \left(\ex[X|\F^n]\right)^2\right]- \ex\left[ X\right]^2 \\
 &\leq &  \ex\left[ \left(X\right)^2\right]- \ex\left[ X\right]^2.
\end{eqnarray*}
We now denote by
$$
\mathcal{P}^n:=\left\{a+ \int Z^n dB^n;\; a\in \R,\, Z^n\in L^2_n(\Omega \times \N) \textnormal{ predictable}\right\}
$$ 
the closed subspace in $L^2(\Omega,\F,P)$, which admits a discrete time predictable integral representation. 
Note that, for every $X\in L^2(\Omega,\F,P)$, $a\in \R$, and $(\F^n_i)_{i\in \N}$-predictable $Z^n\in L^2_n(\Omega \times \N)$,
by the discrete It\^o isometry,
\begin{align*}
 &\ex\left[X\left(a+\int Z^n dB^n\right) \right] = a\ex[X]+\frac{1}{\sqrt{n}}\sum_{i=1}^{\infty}\ex[X\xi^n_i \ex[Z^n_i|\mathcal{F}^n_{i-1}]]  \\
 &=a\ex[X]+ \frac{1}{n} \sum_{i=1}^\infty \ex\left[Z^n_i \sqrt{n} \ex[X\xi^n_i|\F^n_{i-1}] \right]
 =\ex\left[\left(\ex[X]+\int \nabla^nX dB^n \right)\left(a+\int Z^n dB^n\right)\right].
\end{align*}
Hence,
\begin{equation}\label{eq:predictable_projection}
 \pi_{\mathcal{P}^n}X= \ex[X]+\int \nabla^nX dB^n,
\end{equation}
where, for any closed subspace $\mathcal{A}$ in $L^2(\Omega,\F,P)$, $\pi_\mathcal{A}$ denotes the orthogonal projection on 
$\mathcal{A}$.

Our first approximation result for the Clark-Ocone derivative now reads as follows:
\begin{theorem}\label{thm:Clark-Ocone_1}
 Suppose $(X^n)_{n\in \N}$ is a sequence in $L^2(\Omega,\F,P)$ and $X\in L^2(\Omega,\F,P)$. Then, the following are equivalent,
 as $n$ tends to infinity:
 \begin{enumerate}
  \item[(i)] $(\pi_{\mathcal{P}^n}X^n-\ex[X^n])_{n\in \N}$ converges to $X-\ex[X]$ strongly (weakly) in $L^2(\Omega,\F,P)$.
  \item[(ii)] $(\nabla^n_{\lceil n\cdot \rceil} X^n)_{n\in \N}$ converges to $\nabla X$ strongly (weakly) in $L^2(\Omega\times [0,\infty))$.
 \end{enumerate}
A sufficient condition for $(i), (ii)$ is that $(X^n)_{n\in \N}$ converges to $X$ strongly (weakly) in $L^2(\Omega,\F,P)$.
\end{theorem}
\begin{proof}
 Recall that by (\ref{eq:predictable_representation}) and (\ref{eq:predictable_projection})
 \begin{eqnarray*}
  X-\ex[X]&=& \int_0^\infty \nabla X_s dB_s, \\ 
  \pi_{\mathcal{P}^n}X^n- \ex[X^n]&=&\int \nabla^nX^n dB^n.
 \end{eqnarray*}
Hence, Theorem \ref{thm:Ito_integral} provides the equivalence of $(i)$ and $(ii)$. As, for every $g\in \mathcal{E}$, 
$\exp^{\diamond_n}(I^n(\check g^n)) \in \mathcal{P}^n$ by (\ref{eq:Doleans_Dade_discrete}), the sufficient condition is a consequence 
of the following lemma.
\end{proof}

\begin{lemma}\label{lem:projections}
 Suppose that $\mathcal{A}^n, \, n \in \N,$ are closed subspaces of $L^2(\Omega,\F,P)$ such that for every $n\in \N$,
 $$
 \{ \exp^{\diamond_n}(I^n(\check g^n)),\; g\in \mathcal{E}\} \subset \mathcal{A}_n.
 $$
 Then, strong (weak) $L^2(\Omega,\F,P)$-convergence of $(X^n)_{n\in \N}$ to $X$ implies that  
 $(\pi_{\mathcal{A}^n} X^n)_{n\in \N}$ converges to $X$ strongly (weakly) in $L^2(\Omega,\F,P)$ as well.
\end{lemma}
\begin{proof}
 As, for every $g\in \mathcal{E}$,
 $$
 \ex[(\pi_{\mathcal{A}^n} X^n) \exp^{\diamond_n}(I^n(\check g^n)) ] = \ex[X^n\, \pi_{\mathcal{A}^n}( \exp^{\diamond_n}(I^n(\check g^n)))] =
 \ex[X^n \exp^{\diamond_n}(I^n(\check g^n))],
 $$
 we obtain that $(S^n X^n)(\check{g}^n)= (S^n  (\pi_{\mathcal{A}^n}X^n))(\check{g}^n)$. In the case of weak convergence, Theorem 
 \ref{thm_S_transform_convergence_1} now  immediately applies, because
 $$
 \ex\left[\left(\pi_{\mathcal{A}^n} X^n\right)^2\right]\leq \ex\left[\left(X^n\right)^2\right].
 $$
 In the case of strong convergence, we also make use of Theorem \ref{thm_S_transform_convergence_1}, and note that 
 by the already established weak convergence of $(\pi_{\mathcal{A}^n} X^n)_{n\in \N}$ and H\"older's inequality,
 $$
 \ex\left[\left(\pi_{\mathcal{A}^n} X^n\right)^2\right]=\ex\left[X(\pi_{\mathcal{A}^n} X^n)\right]+\ex\left[(X^n-X)(\pi_{\mathcal{A}^n}X^n) \right]
 \rightarrow \ex\left[X^2\right],\quad n\rightarrow \infty.
 $$
\end{proof}

We shall finally discuss an alternative approximation of the generalized Clark-Ocone derivative, which involves 
orthogonal projections 
on appropriate finite-dimensional subspaces. To this end,
 we 
 denote by $\mathcal{H}^n$ the strong closure in $L^2(\Omega, \F,P)$ of the linear span of
\begin{equation*}
\Xi^n:=\left\{\Xi_{A}^n := \prod\limits_{i \in A} \xi^{n}_{i},\quad A \subseteq \N, |A|<\infty\right\},
\end{equation*}
and emphasize that  $\mathcal{H}^n=L^2(\Omega, \F^n,P)$, if and only if the noise distribution of $\xi$ is binary. 
As $\Xi^n$ consists of an orthonormal basis of $\mathcal{H}^n$, every $X^n \in \mathcal{H}^n$ has a unique expansion in terms of this Hilbert space basis, which is called the \emph{Walsh decomposition} of $X^n$,
\begin{equation}\label{Walsh_decomposition}
X^n = \sum\limits_{|A|<\infty} X_{A}^n\Xi_{A}^n,
\end{equation}
where $X^n_{A} =\ex[X^n\Xi_{A}^n]$ satisfies $\sum_{|A|<\infty}(X^n_A)^2 <\infty$. The expectation and $L^2(\Omega,\F,P)$-inner product can be computed in terms of the Walsh decomposition via $\ex[X^n]=X_{\emptyset}^n$ and  
\begin{equation}\label{Walsh_decomposition_inner_product}
\ex\left[X^n Y^n\right] = \sum\limits_{|A|<\infty}X_{A}^nY_{A}^{n},\quad X^n,Y^n \in \mathcal{H}^n,
\end{equation}
cp. \cite{Holden_discrete_Wick}.
A direct computation shows that the Walsh decomposition of a discrete Wick exponential is given by
\begin{equation}\label{eq_discrete_Wick_exp_sum_repr}
\exp^{\diamond_n}(I^n(f^n)) = \sum\limits_{|A|<\infty} \left(n^{-|A|/2}\prod\limits_{i \in A} f^n(i)\right)\Xi_{A}^{n}.
\end{equation}
In view of the M\"obius inversion formula 
 \cite[Theorem 5.5]{Aigner}, we obtain, for every finite subset $B$ of $\N$,
 \begin{equation}\label{eq:inversion}
 \Xi_{B}^n = n^{|B|/2}\, \sum\limits_{C \subseteq B} (-1)^{|B|-|C|} \exp^{\diamond_n}(I^n(\eins_C)).
 \end{equation}
Hence, the set $\{\exp^{\diamond_n}(I^n(\check{g}^n)),\;g\in \mathcal{E} \}$ is total in
 $\mathcal{H}^n$. 

We now consider the finite-dimensional subspaces
$$
\mathcal{H}^n_i:=\textnormal{span}\{\Xi^n_A,\; A\subset \{1,\ldots, i\}\},
$$
and introduce, as a second approximation of the generalized Clark-Ocone derivative, the operator
$$
\bar \nabla^n: L^2(\Omega,\F,P)\rightarrow L^2_n(\Omega\times \N), \quad X \mapsto (\pi_{\mathcal{H}^n_{i-1}}(\nabla^n_i X ))_{i\in \N}.
$$
Notice that
$$
\bar \nabla^n_i X= \sqrt{n}\, \pi_{\mathcal{H}^n_{i-1}}(\xi^n_iX),
$$
if $\xi^n_i X \in L^2(\Omega,\F,P)$.

We are now going to show the following variant of Theorem \ref{thm:Clark-Ocone_1}.
\begin{theorem}\label{thm:Clark-Ocone_2}
 Suppose $(X^n)_{n\in \N}$ is a sequence in $L^2(\Omega,\F,P)$ and $X\in L^2(\Omega,\F,P)$. Then, the following are equivalent,
 as $n$ tends to infinity:
 \begin{enumerate}
  \item[(i)] $(\pi_{\mathcal{H}^n}X^n-\ex[X^n])_{n\in \N}$ converges to $X-\ex[X]$ strongly (weakly) in $L^2(\Omega,\F,P)$.
  \item[(ii)] $(\bar\nabla^n_{\lceil n\cdot \rceil} X^n)_{n\in \N}$ converges to $\nabla X$ strongly (weakly) in $L^2(\Omega\times [0,\infty))$.
 \end{enumerate}
A sufficient condition for $(i), (ii)$ is that $(X^n)_{n\in \N}$ converges to $X$ strongly (weakly) in $L^2(\Omega,\F,P)$.
\end{theorem}
The proof is based on the simple observation that
$\mathcal{H}^n\subset \mathcal{P}^n$, i.e., for every  $X^n\in \mathcal{H}^n$, 
 \begin{equation}\label{eq:predictable_representation_discrete}
 X^n=\ex[X^n]+ \sum_{i=1}^\infty  \nabla^n_i X^n \frac{1}{\sqrt{n}}\xi_i^n.
 \end{equation}
 In order to show this, we recall that  $\{\exp^{\diamond_n}(I^n(\check{g}^n)),\;g\in \mathcal{E} \}$ is total in
 $\mathcal{H}^n$.
 Thus, by continuity of the discretized Clark-Ocone derivative and by the discrete It\^o isometry, it suffices to show (\ref{eq:predictable_representation_discrete})
 in
 the case $X^n=\exp^{\diamond_n}(I^n(\check f^n))$
 for $f\in \mathcal{E}$. A direct computation shows,
 \begin{equation}\label{eq:discrete_clark_ocone_exp}
 \nabla^n_i \exp^{\diamond_n}(I^n(f^n))=f^n(i) \exp^{\diamond_n}(I^n(f^n\eins_{[1,i-1]})),
 \end{equation}
 which in view of (\ref{eq:Doleans_Dade_discrete}) completes the proof of (\ref{eq:predictable_representation_discrete}).

\begin{proof}[Proof of Theorem \ref{thm:Clark-Ocone_2}]
 We first note that, for every $X\in L^2(\Omega,\F,P)$,
 \begin{eqnarray}
 \ex[\pi_{\mathcal{H}^n}X]&=&\ex[X], \label{eq:hilf0025}\\
 \bar \nabla_i^n X&=&\nabla^n_i (\pi_{\mathcal{H}^n}X). \label{eq:hilf0026}
\end{eqnarray}
Indeed, as
$$
\pi_{\mathcal{H}^n}X=\ex[X]+\sum_{1\leq |A|<\infty} \ex[X \Xi^n_A]\Xi^n_A,
$$
Eq. (\ref{eq:hilf0025}) is obvious.  
In order to prove (\ref{eq:hilf0026}), we recall first that $\nabla^n_i(\pi_{\mathcal{H}^n}X)\in \mathcal{H}^n_{i-1}$
(by (\ref{eq:discrete_clark_ocone_exp}) and continuity of the discretized Clark-Ocone derivative) and then note that, for every $A\subset \{1,\ldots,i-1\}$,
 \begin{eqnarray*}
 \ex\left[ \Xi^n_{A} \ex\left[\left.{\xi^n_i} X\right|\mathcal{F}^n_{i-1}\right]\right] &=& \ex[ \Xi^n_{A\cup\{i\}} X]=
 \ex[\Xi^n_{A\cup\{i\}} \pi_{\mathcal{H}^n}(X)]
  \\ &=& \ex\left[\Xi^n_{A} \ex\left[\left.\xi^n_i \pi_{\mathcal{H}^n}(X)\right|\mathcal{F}^n_{i-1}\right]\right]=\ex\left[\Xi^n_A \frac{1}{\sqrt{n}}  \nabla^n_i(\pi_{\mathcal{H}^n}X)\right] .
 \end{eqnarray*}
In particular, by (\ref{eq:predictable_representation_discrete}),  (\ref{eq:hilf0025}), and  (\ref{eq:hilf0026})
\begin{equation}\label{eq:hilf0027}
\pi_{\mathcal{H}^n}X=\ex[X]+\int \bar\nabla^n X dB^n,
\end{equation}
 which is the analogue of (\ref{eq:predictable_projection}). The proof of Theorem \ref{thm:Clark-Ocone_1} can now be repeated verbatim 
 with $\mathcal{P}^n$ replaced by $\mathcal{H}^n$.
\end{proof}

We close this section with two remarks.

\begin{remark}\label{rem:projections}
 In view of Lemma \ref{lem:projections} and the inclusion $\mathcal{H}^n\subset \mathcal{P}^n$ we observe that, for any sequence 
 $(X^n)_{n\in \N}$ in $L^2(\Omega,\F,P)$,
 \begin{eqnarray*}
  && \lim_{n\rightarrow \infty} X_n =X\quad \textnormal{ strongly (weakly) in } L^2(\Omega,\F,P)  \\
  &\Rightarrow & \lim_{n\rightarrow \infty} \pi_{\mathcal{P}^n} X_n =X\quad \textnormal{ strongly (weakly) in } L^2(\Omega,\F,P)
  \\
  &\Rightarrow & \lim_{n\rightarrow \infty} \pi_{\mathcal{H}^n} X_n =X\quad \textnormal{ strongly (weakly) in } L^2(\Omega,\F,P).
 \end{eqnarray*}
In particular, by Theorems \ref{thm:Clark-Ocone_1} and \ref{thm:Clark-Ocone_2}, if the sequence of discretized Clark-Ocone derivatives
$(\nabla^n_{\lceil n\cdot \rceil} X^n)_{n\in \N}$ converges to $\nabla X$ strongly (weakly) in $L^2(\Omega\times [0,\infty))$, then
so does the sequence of modified discretized Clark-Ocone derivatives 
 $(\bar \nabla^n_{\lceil n\cdot \rceil} X^n)_{n\in \N}$ .
\end{remark}

\begin{remark}
The following result can be derived from \cite[Theorem 5 and the examples in Section 5]{Briand_Delyon_Memin} under the additional assumption that 
$\ex[|\xi|^{2+\epsilon}]<\infty$ 
for some $\epsilon>0$ and on a finite time horizon: Strong convergence of $(X^n)_{n\in \N}$ to $X$ in $L^2(\Omega,\F,P)$ implies convergence of the sequence of 
discretized Clark-Ocone derivatives as stated in (ii) of Theorem \ref{thm:Clark-Ocone_1}.  Our Theorem \ref{thm:Clark-Ocone_2} additionally
shows that the conditional expectations $\ex[\cdot|\F^n_{i-1}]$ in the definition of the discretized Clark-Ocone derivative 
can be replaced by the projection on the finite dimensional subspace $\mathcal{H}^n_i$, i.e., if 
$(X^n)_{n\in \N}$ converges to $X$ strongly in $L^2(\Omega,\F,P)$, then
$$
\left(\sqrt{n}\, \pi_{\mathcal{H}^n_{\lceil nt\rceil -1}}(\xi^n_{\lceil nt\rceil}\left( \tau_n(X^n) \right)\right)_{t\in [0,\infty)}
\rightarrow \nabla X 
$$
strongly in $L^2(\Omega\times [0,\infty))$, where $\tau_n$ denotes the truncation at $\pm n$.

We also note that, in view of (\ref{eq:hilf0027}), 
$$
\bar \nabla_i X=\frac{(\pi_{\mathcal{H}^n_i}X)-(\pi_{\mathcal{H}^n_{i-1}}X) }{B^n_{i}-B^n_{i-1}}
$$
can be rewritten as difference operator (where we apply the convention $\frac{\xi^n_i}{\xi^n_i}=1$ when $\xi^n_i$ vanishes). 
This representation shows the close relation to the  weak $L^2(\Omega\times [0,\infty))$-approximation result 
for the generalized Clark-Ocone derivative which is derived in \cite[Corollary 4.1]{Leao_Ohashi}, but for the case of symmetric binary noise only.
\end{remark}

\section{Strong $L^2$-approximation of the chaos decomposition}\label{section_chaos_decomposition}

In this section, we apply Theorem \ref{thm_S_transform_convergence_1} in order to characterize strong $L^2(\Omega,\F,P)$-convergence 
of a sequence $(X^n)$ (where $X^n$ can be represented via multiple Wiener integrals with respect to the discrete time noise 
$(\xi^n_i)_{i\in \N}$) via convergence of the coefficient functions of such a discrete chaos decomposition. 

Recall first, that every $X \in L^2(\Omega, \F, P)$ has a unique  Wiener chaos decomposition in terms of multiple Wiener integrals
\begin{align}
X = \sum\limits_{k=0}^{\infty} I^k(f^k_X),\label{eq_Wiener_chaos_decomp} 
\end{align}
where $f^k_X \in \widetilde{L^{2}}([0,\infty)^k)$,  see e.g. \cite[Theorem 1.1.2]{Nualart}.
Here, we denote by $L^{2}([0,\infty)^k)$ the Hilbert space 
 of square-integrable functions with respect to the $k$-dimensional Lebesgue measure
 and by $\widetilde{L^{2}}([0,\infty)^k)$ the subspace of functions in $L^{2}([0,\infty)^k)$ 
 which are symmetric in the $k$ variables.
We apply the standard convention $\widetilde{L^{2}}([0,\infty)^0)
={L^{2}}([0,\infty)^0)=\R$, $I^0(f^0)=f^0$, and recall
that, for $k\geq 1$ and $f^k\in\widetilde{L^{2}}([0,\infty)^k)$, the multiple Wiener integral can be defined as iterated It\^o integral:
$$
I^k(f^k)=k! \int_0^\infty \int_0^{t_k} \cdots \int_0^{t_2} f^k(t_1,\ldots, t_k) dB_{t_1} \cdots dB_{t_{k-1}} dB_{t_k}. 
$$
The It\^o isometry therefore immediately implies the following well-konwn Wiener-It\^o isometry for multiple Wiener integrals,
\begin{equation}\label{eq:multiple_Wiener_isometry}
 \ex[I^{k}(f^{k}) \,I^{k'}(g^{k'})] = \delta_{k,k'} \ k! \ \langle {f^{k}}, {g^{k}}\rangle_{L^{2}([0,\infty)^k}
\end{equation}
for functions $f^k\in \widetilde{L^{2}}([0,\infty)^k)$ and $g^{k'}\in \widetilde{L^{2}}([0,\infty)^{k'})$.

The main theorem of this section now reads as follows:

\begin{theorem}[Wiener chaos limit theorem]\label{thm_S_transform_convergence_2}
Suppose $(X^n)_{n\in \N}$ is a sequence in $L^2(\Omega,\F,P)$. Then the following assertions are equivalent as $n$ tends to infinity:
\begin{enumerate}
\item[(i)] The sequence $(\pi_{\mathcal{H}^n}X^n)$ converges strongly in $L^2(\Omega, \F, P)$.
\item[(ii)] For every $k \in \N_0$, the sequence $(\widehat{f^{n,k}_{X^n}})_{n \in \N}$, defined via
\begin{align}\label{eq_thm_S_transform_convergence_2_integrand}
\widehat{f^{n,k}_{X^n}}(u_1, \ldots, u_k) := \left\{ \begin{array}{cl} 
\ex\left[X^n \frac{n^{k/2}}{k!} \Xi^n_{\{\lceil n u_1\rceil, \ldots, \lceil n u_k\rceil\}}\right], & 
|\{\lceil n u_1\rceil, \ldots, \lceil n u_k\rceil\}\cap \N|=k, \\ 0, & \textnormal{otherwise,}
\end{array}\right.
\end{align}
is strongly convergent in $L^2([0,\infty)^k)$ and 
\begin{equation}\label{eq:limlimsup}
\lim\limits_{m\rightarrow \infty} \limsup\limits_{n \rightarrow \infty} \sum\limits_{k=m}^{\infty} k! \| \widehat{f^{n,k}_{X^n}}\|^2_{{L^2}([0,\infty)^k)} =0. 
\end{equation}
\end{enumerate}
In this case, the limit $X$ of $(\pi_{\mathcal{H}^n}X^n)_{n\in \N}$ has the Wiener chaos decomposition $X = \sum\limits_{k=0}^{\infty} I^k(f^k_X)$ with $f^k_X = 
\lim\limits_{n \rightarrow \infty} \widehat{f^{n,k}_{X^n}}$ in $L^2([0,\infty)^k)$.
\end{theorem}
We recall that, by Remark \ref{rem:projections}, the strong $L^2(\Omega,\mathcal{F},P)$-convergence of $(X^n)$ to $X$ is 
a sufficient condition for the strong approximation of the chaos coefficients of $X$ as stated in
(ii)  of the above theorem.

Before proving  Theorem \ref{thm_S_transform_convergence_2}, we briefly discuss this result. To this end, we first  recall the relation between
Walsh decomposition and discrete chaos decomposition.
The discrete multiple Wiener integrals are defined analogously to the continuous setting, see e.g. \cite[Section 1.3]{Privault_book}.
For all $k,n \in \N$ we consider the Hilbert space 
\[
L^2_n(\N^k) := \left\{f^{n,k}: \N^k \rightarrow \R : \,\sum_{(i_1, \ldots, i_k) \in \N^k} \left(f^{n,k}(i_1,\ldots, i_k)\right)^2<\infty\right\} 
\]
endowed with the inner product 
\[
\langle f^{n,k}, g^{n,k} \rangle_{L^2_n(\N^k)} := n^{-k}\sum_{(i_1, \ldots, i_k) \in \N^k} f^{n,k}(i_1,\ldots, i_k) g^{n,k}(i_1,\ldots, i_k). 
\]
The closed subspace of symmetric functions in $L^2_n(\N^k)$ which vanish on the diagonal part
$$
\partial_k:=\left\{(i_1,\ldots,i_k)\in \N^k:\;  | \{i_1,\ldots,i_k \}|<k \right\}
$$
is denoted by $\widetilde{L^2_n}(\N^k)$.

 Then, for $k\in \N$,
the discrete multiple Wiener integral of $f^{n,k}\in \widetilde{L^2_n}(\N^k)$ with respect to the random walk $B^n$ is defined as
\[
I^{n,k}(f^{n,k}) := n^{-k/2} k!\sum_{\substack{(i_1, \ldots, i_k) \in \N^k, \ i_1<\cdots<i_k}} f^{n,k}(i_1, \ldots, i_k) \Xi^n_{\{i_1,\ldots, i_k\}}. 
\] 
We notice that $I^{n,k}$ is linear on $\widetilde{L^2_n}(\N^k)$ and fulfills $\ex[I^{n,k}(f^{n,k})]=0$ as well as
the isometry 
\begin{align}\label{eq_discrete_multiple_Wiener_covariance}
\ex[I^{n,k}(f^{n,k}) I^{n,k'}(g^{n,k'})] &= \delta_{k,k'} \ k! \ \langle {f^{n,k}}, {g^{n,k}}\rangle_{L^2_n(\N^k)}
\end{align}
for $f^{n,k} \in \widetilde{L^2_n}(\N^k)$, $g^{n,k'} \in  \widetilde{L^2_n}(\N^{k'})$ and possibly different orders $k, k' \in \N$.
As in the continuous time setting, we apply the convention that $I^{n,0}$ is the identity on $\widetilde{L^2_n}(\N^0):=\R$, 
and refer to \cite[Section 1.3]{Privault_book} for further properties of such discrete multiple Wiener integrals.

We now fix $X^n\in \mathcal{H}^n$. In view of the Walsh decomposition $X^n = \sum_{|A|<\infty} \ex[X^n\Xi_{A}^n]\Xi_{A}^n$, we observe that 
the discrete analog of the Wiener chaos decomposition
\begin{align}
X^n &= \sum\limits_{k=0}^{\infty} n^{-k/2}k! \sum\limits_{\substack{(i_1, \ldots, i_k) \in \N^k, i_1<\cdots <i_k}} \frac{n^{k/2}}{k!}X^n_{\{i_1,\ldots, i_k\}} \Xi^n_{\{i_1,\ldots, i_k\}} =\sum\limits_{k=0}^{\infty} I^{n,k}(f^{n,k}_{X^n}),\label{eq_discrete_chaos_decomposition}
\end{align}
holds for the integrands $f^{n,k}_{X^n} \in \widetilde{L^2_n}(\N^k)$ given by
\begin{align}\label{eq_discrete_chaos_integrand}
f^{n,k}_{X^n}(i_1, \ldots, i_k) := \left\{ \begin{array}{cl} \ex\left[\frac{n^{k/2}}{k!} X^n \Xi^n_{\{i_1, \ldots, i_k\}}\right], & |\{i_1, \ldots, i_k\}\cap \N|=k \\ 
0, & \textnormal{otherwise}. \end{array}\right.
 \end{align}
Hence, this discrete analog of the Wiener chaos decomposition \eqref{eq_Wiener_chaos_decomp}
for random variables in $\mathcal{H}^n$ is just a reformulation of the Walsh decomposition \eqref{Walsh_decomposition}.

Given a general element $f^{n,k} \in \widetilde{L^2_{n}}(\N^k)$ we define its embedding into simple continuous time functions in $k$ variables as
 \begin{align}
 \widehat{f^{n,k}}(u_1, \ldots, u_k) &:= f^{n,k}\left(\lceil n u_1\rceil, \ldots, \lceil n u_k\rceil\right)\nonumber\\  
&= \sum\limits_{i_1, \ldots, i_k =1}^{\infty} f^{n,k}(i_1,\ldots, i_k) \eins_{(\frac{i_1-1}{n}, \frac{i_1}{n}] \times \cdots \times (\frac{i_k-1}{n}, \frac{i_k}{n}]}(u_1, \ldots, u_k),\label{eq_embedding_of_discrete_function}
 \end{align}
which is consistent with the notation already applied in (\ref{eq_thm_S_transform_convergence_2_integrand}) and (\ref{eq_discrete_chaos_integrand}).
Here and in what follows, we apply the convention that $f^{n,k}$ vanishes when one of its arguments is set to zero.

We can now rephrase Theorem \ref{thm_S_transform_convergence_2} in the following way: \\[0.2cm]
{\it The sequence $(X^n)$, with $X^n\in \mathcal{H}^n$, converges to $X$ strongly in $L^2(\Omega,\F,P)$, if and only if, for all orders $k\in \N_0$, the sequence of coefficient functions 
of the discrete chaos decomposition of $X^n$ converge (after the natural embedding into continuous time) to the coefficient functions 
of the Wiener chaos of $X$ strongly in $L^2([0,\infty)^k)$ and the tail condition (\ref{eq:limlimsup}) is satisfied.}
\\ 
\begin{remark}\label{rem_main_theorem}
 Convergence of discrete multiple Wiener integrals to continuous multiple Wiener integrals was studied in \cite{Surgailis}  as a main tool 
for proving noncentral limit theorems. The results in Section 4 of the latter reference imply that, for every $k\in \N_0$,
the sequence of discrete multiple Wiener integrals $(I^{n,k}(f^{n,k}))_{n\in \N}$ converges in distribution
to the multiple Wiener integral $I^k(f^k)$, if $(\widehat{f^{n,k}})_{n\in \N}$ converges to $f^k$ strongly in $L^2([0,\infty)^k)$.
Our result lifts this convergence in distribution to strong $L^2(\Omega,\F,P)$-convergence and adds the converse:
$$
L^2(\Omega,\F,P)\textnormal{-}\lim_{n\rightarrow \infty} I^{n,k}(f^{n,k})=I^k(f^k) \; \Leftrightarrow \; L^2([0,\infty)^k)\textnormal{-}\lim_{n\rightarrow \infty}  \widehat{f^{n,k}}=f^k.
$$
We note that the $L^2(\Omega,\F,P)$-convergence of the sequence $(I^{n,k}(f^{n,k}))$ even implies convergence in 
$L^p(\Omega,\F,P)$ for $p>2$, if $\ex[|\xi|^r]<\infty$ for some $r>p$. Indeed, in this case, the sequence  $(|I^{n,k}(f^{n,k})|^p)$ 
is uniformly integrable by the hypercontractivity inequality of \cite{KS} in the variant of \cite[Proposition 5.2]{BT}.
\end{remark}

The following elementary corollary of Theorem \ref{thm_S_transform_convergence_2} generalizes Proposition \ref{proposition_Wiener_and_Wick_exp}. It makes use of the fact 
that the chaos decompositions of (discrete) Wick exponentials are given,
for all $f \in L^2([0,\infty))$, $f^n \in L^2_n(\N)$, by
\begin{equation}\label{eq:exponentials_chaos}
\exp^{\diamond}(I(f)) = \sum_{k=0}^\infty I^k(\frac{1}{k!} f^{{\otimes} k}), \quad 
\exp^{\diamond_n}(I^n(f^n)) = \sum_{k=0}^{\infty} I^{n,k}(\frac{1}{k!}(f^n)^{{\otimes} k}\eins_{\partial_k^c}).
\end{equation}
For a proof of the continuous case see e.g. \cite[Theorem 3.21, Theorem 7.26]{Janson}. The statement
of the discrete case is a direct consequence of \eqref{eq_discrete_Wick_exp_sum_repr}.

\begin{corollary}
 Suppose $f\in L^2([0,\infty))$ and $(f^n)$ is a sequence with $f^n\in L^2_n(\N)$ for every $n \in \N$. Then, as $n$ tends to infinity
 (in the sense of strong convergence),
 \begin{align*}
\widehat{f^n} \rightarrow f \textnormal{ in } L^2([0,\infty)) &\Leftrightarrow  I^n(f^n) \rightarrow I(f) \textnormal{ in } L^2(\Omega,\mathcal{F},P)\\
&\Leftrightarrow \exp^{\diamond_n}(I^n(f^n)) \rightarrow \exp^{\diamond}(I(f)) \textnormal{ in } L^2(\Omega,\mathcal{F},P) .
 \end{align*}
   \end{corollary}
\begin{proof}
 In view of Theorem \ref{thm_S_transform_convergence_2} and (\ref{eq:exponentials_chaos}), we only have to show that 
 $ \widehat{f^n}\rightarrow f \textnormal{ strongly in } L^2([0,\infty))$
 implies that  $ \widehat{(f^n)^{\otimes k}\eins_{\partial_k^c}}\rightarrow f^{\otimes k} \textnormal{ strongly in } L^2([0,\infty)^k)$, for every $k\geq 2$.
 This is a consequence of the following lemma.
 \end{proof}

 \begin{lemma}\label{lem:diagonal}
(i) Fix $k\in \N_0$. Suppose $(f^{n,k})_{n\in\N}$ is a sequence such that $f^{n,k}\in L^2_n(\N^k)$ for every $n \in \N$ and $(\widehat{f^{n,k}})$ converges 
to some $f^k$ strongly in $L^2([0,\infty)^k)$. Then, the sequence $(\widehat{f^{n,k} \eins_{\partial_k^c}})$ converges to $f^k$ strongly in $L^2([0,\infty)^k)$
as well. \\[0.1cm]
(ii) Suppose $(f^{n})_{n\in\N}$ is a sequence such that $f^{n}\in L^2_n(\N)$ for every $n \in \N$ and $(\widehat{f^{n}})$ converges 
to some $f$ strongly in $L^2([0,\infty))$. Then, for every $k\geq 2$,  the sequences $(\widehat{(f^{n})^{\otimes k} })$ and  $(\widehat{(f^{n})^{\otimes k}\eins_{\partial_k^c} })$
converge to $f^{\otimes k}$ strongly in $L^2([0,\infty)^k)$.
 \end{lemma}
 \begin{proof}
 (i) We decompose,
 $$
 \|\widehat{f^{n,k}\eins_{\partial_k^c}}- f^{k}\|_{L^2([0,\infty)^k)}\leq  \|\widehat{f^{n,k}\eins_{\partial_k^c}} 
 -\widehat{f^{n,k}}\|_{L^2([0,\infty)^k) }+
 \|\widehat{f^{n,k}}- f^{ k}\|_{L^2([0,\infty)^k)} .
 $$
 The second term goes to zero by assumption. The first one equals
$$
\left(\int_{[0,\infty)^k}  |f^{n,k} (\lceil nu_1\rceil,\ldots,\lceil nu_k\rceil)|^2 \eins_{\{|\{\lceil nu_1\rceil,\ldots, \lceil nu_k\rceil \}|<k \}}  \right)^{1/2}.
$$
The sequence of integrands tends to 0 almost everywhere, because 
$$\lim_{n\rightarrow \infty} \eins_{\{|\{\lceil nu_1\rceil,\ldots, \lceil nu_k\rceil \}|<k \}}=\eins_{\{\, u_l=u_p,\;\textnormal{ for some }l\neq p\}}.$$
Moreover, the sequence of integrands inherits uniform integrability from the $L^2([0,\infty)^k)$-convergent series 
$(\widehat{f^{n,k}})$. Therefore, the first term goes to zero by interchanging limit and integration. \\[0.1cm]
(ii) As tensor powers commute with discretization and embedding, i.e. 
\begin{align}
(\check{g}^n)^{{\otimes} k} = (\check{(g)^{{\otimes} k}})^n, \qquad \widehat{h^n}^{{\otimes} k} = \widehat{(h^n)^{{\otimes} k}} \label{eq_symmetrization_embedding_discretization}
\end{align}
for all $k \in \N$, $g \in \mathcal{E}$, $h^n \in L^2_n(\N)$, and as the tensor product is continuous, 
we observe inductively that  $ \widehat{(f^n)^{\otimes k}}\rightarrow f^{\otimes k}$  strongly in  $L^2([0,\infty)^k)$. 
Then, for the second sequence, part (i) applies.
\end{proof}

We now start to prepare the proof of Theorem \ref{thm_S_transform_convergence_2}.

\begin{proposition}\label{proposition_discrete_S_transform_mWi_as_integral}
Let $k\in \N_0$. Then, for all $g \in \mathcal{E}$ and sequences $(f^{n,k})_{n\in \N}$ such that $f^{n,k} \in \widetilde{L^2_n}(\N^k)$ and $\sup_{n \in \N}\|f^{n,k}\|_{L^2_n(\N)} < \infty$,
$$
\lim_{n\rightarrow \infty} \left|(S^n I^{n,k}(f^{n,k}))(\check{g}^n) - (S I^{k}(\widehat{f^{n,k}}))(g) \right|=0.
$$

\end{proposition}

\begin{proof}
First note that, by (\ref{eq_discrete_multiple_Wiener_covariance}), (\ref{eq:exponentials_chaos}), and as $f^{n,k}$ vanishes on the diagonal $\partial_k$,
\begin{align*}
\displaybreak[0]
&(S^n \ I^{n,k}(f^{n,k}))(\check{g}^n) = \ex\left[I^{n,k}(f^{n,k}) \exp^{\diamond_n}(I^n(\check{g}^n))\right]
 = \langle f^{n,k}, (\check{g}^n)^{{\otimes} k} \rangle_{L^2_{n}(\N^k)}\nonumber\\
&= \int_{[0,\infty)^k} \widehat{f^{n,k}}(x_1, \ldots, x_k)  \widehat{(\check{g}^n)^{{\otimes} k}}(x_1, \ldots, x_k) dx_1 \cdots dx_k.\label{eq_S_transform_of_multiple_Wiener_integral}
\end{align*}
Analogously, making use of the Wiener-It\^o isometry for the continuous chaos decomposition (\ref{eq:multiple_Wiener_isometry}) 
instead of (\ref{eq_discrete_multiple_Wiener_covariance}), we get
$$
(S I^{k}(\widehat{f^{n,k}}))(g)=\int_{[0,\infty)^k} \widehat{f^{n,k}}(x_1, \ldots, x_k)  g^{{\otimes} k}(x_1, \ldots, x_k) dx_1 \cdots dx_k.
$$
Hence, by the Cauchy-Schwarz inequality, we conclude
\begin{align*}
\displaybreak[0]
&\left|(S^n I^{n,k}(f^{n,k}))(\check{g}^n) -  (S I^{k}(\widehat{f^{n,k}}))(g)\right|\\ 
\displaybreak[0]
&= \left|\int_{[0,\infty)^k} \widehat{f^{n,k}}(x_1, \ldots, x_k)  \left(\widehat{(\check{g}^n)^{{\otimes} k}} -g^{{\otimes} k}\right)(x_1, \ldots, x_k) dx_1 \cdots dx_k \right|\\ 
\displaybreak[0]
&\leq \left(\sup_{m \in \N}\|f^{m,k}\|_{L^2_n(\N)}\right)^{1/2}\|g^{{\otimes} k}- \widehat{(\check{g}^n)^{{\otimes} k}}\|_{L^2([0,\infty)^k)},
\end{align*}
which tends to zero for $n\rightarrow \infty$ by Lemma \ref{lem:diagonal}.
\end{proof}

\begin{corollary}\label{cor_observation_Wick_powers_simple_functions}
 Suppose $g\in \mathcal{E}$. Then, for every $k\in\N$, 
 \begin{align*}
I^{n,k}((\check{g}^n)^{{\otimes} k} \eins_{\partial_k^c}) \rightarrow I^k(g^{{\otimes} k})
\end{align*}
strongly in $L^2(\Omega, \F, P)$. 
\end{corollary}
\begin{proof}
 We check item (ii) in Theorem \ref{thm_S_transform_convergence_1}. To this end, we decompose, for every $g,h\in \mathcal{E}$,
 \begin{eqnarray*}
 && \left|(S^n I^{n,k}((\check{g}^n)^{{\otimes} k} \eins_{\partial_k^c}))(\check{h}^n) -(S I^k(g^{{\otimes} k}))(h)  \right|\nonumber \\
 &\leq&  \left|(S^n I^{n,k}((\check{g}^n)^{{\otimes} k}\eins_{\partial_k^c}))(\check{h}^n) -(S I^k(\widehat{(\check{g}^n)^{{\otimes} k}\eins_{\partial_k^c}}))(h)  \right|
 \\ &&+ 
 \left|(S I^k(\widehat{(\check{g}^n)^{{\otimes} k}\eins_{\partial_k^c}}))(h) -(S I^k(g^{{\otimes} k}))(h)  \right|. 
 \end{eqnarray*}
The first term on the righthand side tends to zero by Proposition \ref{proposition_discrete_S_transform_mWi_as_integral}.
The second one equals, by the isometry for multiple Wiener integrals,
$$
\int_{[0,\infty)^k} h^{\otimes k}(x)\left(\widehat{(\check{g}^n)^{{\otimes} k}\eins_{\partial_k^c}}-g^{\otimes k}\right)(x) dx
$$
and goes to zero by Lemma \ref{lem:diagonal}. Consequently,
\[
\lim_{n\rightarrow\infty} (S^n I^{n,k}((\check{g}^n)^{{\otimes} k}\eins_{\partial_k^c}))(\check{h}^n) =(S I^k(g^{{\otimes} k}))(h) 
\]
for all $k \in \N_0$ and $g, h \in \mathcal{E}$. For $h=g$, this implies 
$\ex[I^{n,k}((\check{g}^n)^{{\otimes} k}\eins_{\partial_k^c})^2]\rightarrow \ex[I^k(g^{{\otimes} k})^2]$
by the orthogonality of (discrete) multiple Wiener integrals of different orders. Thus, Theorem \ref{thm_S_transform_convergence_1}
applies.
\end{proof}

We are now in the position to give the proof of Theorem \ref{thm_S_transform_convergence_2}.

\begin{proof}[Proof of Theorem \ref{thm_S_transform_convergence_2}]
`$(i) \Rightarrow (ii)$': We denote the limit of $(\pi_{\mathcal{H}^n} X^n)_{n\in\N}$ in $L^2(\Omega,\mathcal{F},P)$ by $X$ and recall that
 $$
 \pi_{\mathcal{H}^n} X^n=\sum_{|A|<\infty} \ex[X^n \Xi^n_A]\Xi^n_A =  \sum\limits_{k=0}^{\infty} I^{n,k}(f_{X^n}^{n,k}),
 $$
 with $f^{n,k}_{X^n}$ as defined in (\ref{eq_discrete_chaos_integrand}).
Throughout  the proof we omit the subscripts from the coefficients of the chaos decompositions and write 
$\pi_{\mathcal{H}^n} X^n =\sum\limits_{k=0}^{\infty} I^{n,k}(f^{n,k})$ and $X = \sum\limits_{k=0}^{\infty} I^k(f^k)$.
Thanks to Corollary \ref{cor_observation_Wick_powers_simple_functions} and the 
orthogonality of (discrete) multiple Wiener integrals of different orders, we obtain,  for every $k \in \N_0$,
\begin{align*}
(S^n I^{n,k}(f^{n,k}))(\check{g}^n)  &= 
\frac{1}{k!}\ex[\pi_{\mathcal{H}^n}(X^n) I^{n,k}((\check{g}^n)^{{\otimes} k}\eins_{\partial_k^c})] 
 \rightarrow \frac{1}{k!}\ex[X I^k(g^{{\otimes} k})] = (S I^{k}(f^{k}))(g).
\end{align*}
The estimate $\sup_{n \in \N} \ex[(I^{n,k}(f^{n,k}))^2] \leq \sup_{n \in \N} \ex[(\pi_{\mathcal{H}^n}X^n)^2]<\infty $ now yields,
in view of Theorem \ref{thm_S_transform_convergence_1},  weak $L^2(\Omega,\F, P)$-convergence of $(I^{n,k}(f^{n,k}))_{n\in\N}$ towards $I^{k}(f^{k})$. 
As $\pi_{\mathcal{H}^n}X^n \rightarrow X$ strongly in $L^2(\Omega, \F, P)$, we thus obtain
\begin{align}
\displaybreak[0]
\ex[(I^{n,k}(f^{n,k}))^2] = \ex[I^{n,k}(f^{n,k}) \pi_{\mathcal{H}^n}X^n] 
&\rightarrow \ex[I^{k}(f^{k}) X] = \ex[(I^{k}(f^{k}))^2].\label{eq_Thm_S_transform_convergence_2_1}
\end{align}
Hence,  $I^{n,k}(f^{n,k}) \rightarrow I^{k}(f^{k})$ strongly in $L^2(\Omega, \F, P)$ for all $k \in \N_0$ by Theorem \ref{thm_S_transform_convergence_1}. 
Due to the isometries \eqref{eq:multiple_Wiener_isometry} and \eqref{eq_discrete_multiple_Wiener_covariance}, this implies 
\begin{align}\label{eq_Thm_S_transform_convergence_2_2}
k!\| \widehat{f^{n,k}} \|^2_{{L^2}([0,\infty)^k)} = \| I^{n,k}(f^{n,k})\|^2_{L^2(\Omega,\mathcal{F},P)}\rightarrow \| I^{k}(f^{k})\|^2_{L^2(\Omega,\mathcal{F},P)} = k!\| f^{k}\|^2_{{L^2}([0,\infty)^k)}. 
\end{align}
Moreover,
 for every $g \in \mathcal{E}$, we obtain
\begin{align*}
&\langle g^{{\otimes} k}, \widehat{f^{n,k}} - f^k\rangle_{{L^2}([0,\infty)^k)} = (S I^{k}(\widehat{f^{n,k}}))(g) - (S I^{k}(f^{k}))(g)\\
&= \left((S I^{k}(\widehat{f^{n,k}}))(g) - (S^n I^{n,k}(f^{n,k}))(\check{g}^n)\right) + \ex\left[I^{n,k}(f^{n,k}) \exp^{\diamond_n}(I^n(\check{g}^n))-I^{k}(f^{k})\exp^{\diamond}(I(g))\right] \\ &\rightarrow 0,
\end{align*}
by Propositions \ref{proposition_Wiener_and_Wick_exp} and \ref{proposition_discrete_S_transform_mWi_as_integral}, and the 
$L^2(\Omega, \F, P)$-convergence of $(I^{n,k}(f^{n,k}))_{n\in\N}$ to $I^{k}(f^{k})$.
Since the set $\{g^{{\otimes} k}, g \in \mathcal{E}\}$ is total in $\widetilde{L^2}([0,\infty)^k)$, we may conclude 
that $(\widehat{f^{n,k}}) $ converges weakly in $\widetilde{L^2}([0,\infty)^k)$ to $f^k$  by \cite[Theorem V.1.3]{Yosida}. Finally,  
\eqref{eq_Thm_S_transform_convergence_2_2} turns this weak convergence into strong ${L^2}([0,\infty)^k)$-convergence. In particular,
the $k$th coefficient in the chaos decomposition of the limiting random variable $X$ is the strong ${L^2}([0,\infty)^k)$-limit of 
$(\widehat{f^{n,k}}) $, as asserted.

It remains to show (\ref{eq:limlimsup}). However, by \eqref{eq_Thm_S_transform_convergence_2_1} and the isometries
for (discrete) multiple Wiener integrals,
\begin{align*}
\lim\limits_{n \rightarrow \infty} \sum\limits_{k=m}^{\infty} k!\| \widehat{f^{n,k}} \|^2_{{L^2}([0,\infty)^k)}
&= \lim\limits_{n \rightarrow \infty} \sum\limits_{k=m}^{\infty} \| I^{n,k}(f^{n,k})\|_{L^2(\Omega,\mathcal{F},P)}^2\\ 
&= \lim\limits_{n \rightarrow \infty} \left(\|\pi_{\mathcal{H}^n}X^n\|^2_{L^2(\Omega,\mathcal{F},P)} - \sum\limits_{k=0}^{m-1}\|I^{n,k}(f^{n,k})\|^2_{L^2(\Omega,\mathcal{F},P)}\right)\\
&= \|X\|^2_{L^2(\Omega,\mathcal{F},P)} - \sum\limits_{k=0}^{m-1}\|I^{k}(f^{k})\|^2_{L^2(\Omega,\mathcal{F},P)} \rightarrow 0 
\end{align*}
as $m$ tends to infinity.
\\[0.2cm]
`$(ii) \Rightarrow (i)$': In order to lighten the notation, we again denote the function $f_{X^n}^{n,k}$ from 
\eqref{eq_discrete_chaos_integrand} by $f^{n,k}$. Assuming (ii), the strong $L^2([0,\infty)^k)$-limit of $\widehat{ f^{n,k}}$ exists
and will be denoted $f^k$. We first show that $(I^{n,k}(f^{n,k}))$ converges to $I^{k}(f^{k})$ strongly in $L^2(\Omega,\mathcal{F},P)$ for all $k \in\N_0$
by means of Theorem \ref{thm_S_transform_convergence_1}. To this end, we observe that, for every $g\in \mathcal{E}$,
\begin{eqnarray*}
 (S^n\, I^{n,k}(f^{n,k}))(\check{g}^n)&=& \left((S^n\, I^{n,k}(f^{n,k}))(\check{g}^n)- (S\, I^{k}(\widehat{f^{n,k}}))({g})\right)+(S\, I^{k}(\widehat{f^{n,k}}))({g})
 \\ &\rightarrow& (S\, I^{k}({f^{k}}))({g})
\end{eqnarray*}
by Proposition \ref{proposition_discrete_S_transform_mWi_as_integral} and the isometry for continuous multiple Wiener integrals.
Moreover, again, by the isometries for discrete and continuous multiple Wiener integrals,
\begin{eqnarray*}
 \ex\left[(I^{n,k}(f^{n,k}))^2\right]=k!\| f^{n,k}\|^2_{L^2_n(\N^k)}=k!\|\widehat{f^{n,k}}\|^2_{L^2([0,\infty)^k)}\rightarrow 
 k!\|{f^{k}}\|^2_{L^2([0,\infty)^k)}=  \ex\left[(I^{k}(f^{k}))^2\right].
\end{eqnarray*}
So, Theorem \ref{thm_S_transform_convergence_1} applies indeed. With the $L^2(\Omega,\mathcal{F},P)$-convergence of $I^{n,k}(f^{n,k})$ to $I^{k}(f^{k})$
at hand, we can now decompose, for every $m\in \N$,
\begin{eqnarray}\label{eq:hilf0002}
&& \limsup_{n\rightarrow \infty} \ex\left[\left| \pi_{\mathcal{H}^n} X^n-\sum_{k=0}^\infty  I^k(f^k)\right|^2\right]\nonumber\\
&\leq& 3\limsup_{n\rightarrow \infty} \left(\| \sum\limits_{k=0}^{m-1} I^k(f^k) -\sum\limits_{k=0}^{m-1} I^{n,k}(f^{n,k})\|_{L^2(\Omega,\mathcal{F},P)}^2 + 
\sum\limits_{k=m}^{\infty} \| I^k(f^k) \|_{L^2(\Omega,\mathcal{F},P)}^2\right. \nonumber \\
&&\qquad \qquad \qquad \left. + \sum\limits_{k=m}^{\infty} \| I^{n,k}(f^{n,k})\|_{L^2(\Omega,\mathcal{F},P)}^2\right)
\nonumber \\
&=& 3  \sum\limits_{k=m}^{\infty} k! \| f^k \|_{L^2([0,\infty)^k)}^2 +3 \limsup_{n\rightarrow \infty} \sum\limits_{k=m}^{\infty} k!
\| \widehat{f^{n,k}}\|_{L^2([0,\infty)^k)}^2.
\end{eqnarray}
By Fatou's lemma,
\begin{align*}
\sum\limits_{k=m}^{\infty} k! \| f^k \|_{L^2([0,\infty)^k)}^2 = \sum\limits_{k=m}^{\infty} k! 
\lim_{n\rightarrow \infty} \| \widehat{f^{n,k}}\|_{L^2([0,\infty)^k)}^2
\leq \liminf_{n \rightarrow \infty}\sum\limits_{k=m}^{\infty} k!\| \widehat{f^{n,k}}\|_{L^2([0,\infty)^k)}^2.
\end{align*}
Hence, letting $m$ tend to infinity in (\ref{eq:hilf0002}), we observe, thanks to (\ref{eq:limlimsup}), that $(\pi_{\mathcal{H}^n} X^n)$ converges strongly in $L^2(\Omega,\mathcal{F},P)$.
\end{proof}

We close this section with an example.
\begin{example}\label{example:Wiener_integrals}
 Fix $X\in L^2(\Omega, \mathcal{F},P)$. Theorem \ref{thm_S_transform_convergence_2} with $X^n=X$ for every $n\in \N$,
 implies that the chaos coefficients $f_X^k$, $k\in \N_0$, of $X$ are given as the strong $L^2([0,\infty)^k)$-limit
 of 
 $$
 \widehat{f^{n,k}}(u_1, \ldots, u_k) := 
\frac{1}{k!} \,\ex\left[X   \left(\prod_{l=1}^k \frac{B^n_{\lceil n u_l\rceil}-B^n_{(\lceil n u_l\rceil-1)}}{1/n}\right) \right]\;  \eins_{\{|\{\lceil n u_1\rceil, \ldots, \lceil n u_k\rceil\}\cap \N|=k\}}.
 $$
 This formula can be further simplified when $X$ is $\F_T$-measurable. Then,  one can show, analogously to Example \ref{ex:conditional} (ii),
 that the sequence $(\pi_{\mathcal{H}^n_{\lfloor nT\rfloor}}X)$ converges to $X$ strongly in $L^2(\Omega,\F,P)$. Applying 
 Theorem \ref{thm_S_transform_convergence_2} with the latter sequence, shows that
 the chaos coefficients $f_X^k$, $k\in \N_0$, are the strong $L^2([0,\infty)^k)$-limit of
 $$
 \widehat{f^{n,k}}(u_1, \ldots, u_k) := 
\frac{1}{k!} \,\ex\left[X   \left(\prod_{l=1}^{k} \frac{B^n_{\lceil n u_l\rceil}-B^n_{(\lceil n u_l\rceil-1)}}{1/n}\right) \right]\;  \eins_{\{|\{\lceil n u_1\rceil, \ldots, \lceil n u_k\rceil\}\cap \{1,\ldots,\lfloor nT\rfloor \}|=k\}}.
 $$
 In this case, for each fixed $n\in \N$, only finitely many of the functions $\widehat{f^{n,k}}$, $k\in \N_0$, are not constant zero, and these are simple
 functions with finitely many steps sizes only.

 These two approximation formulas for the chaos coefficients of $X$ are one way to give a rigorous meaning of the heuristic formula
 $$
 f^k_X(u_1,\ldots, u_k)=\frac{1}{k!} \ex\left[X \left(\prod_{l=1}^k  \dot B_{u_l}\right)\right],
 $$
 where $\dot B$ is white noise, which is called Wiener's intuitive recipe in \cite{Cutland_Ng}. The latter paper provides
 another rigorous meaning to Wiener's recipe via nonstandard analysis, which is closely related to our approximation formulas in the special 
 case of symmetric Bernoulli noise. The authors show that
 $$
 f_X^k(^\circ t_1,\ldots,^\circ t_k) =\frac{1}{k!}\;^\circ\ex\left[x(b) \left(\frac{\Delta b_{t_1}}{\Delta t} \cdots \frac{\Delta b_{t_k}}{\Delta t}\right)\right],\quad t_l\in T=\{j \Delta t,\;0\leq j < N^2\},
$$
where $N$ is infinite, $\Delta t=1/N$, $b_t(\omega)= \sqrt{\Delta t} \sum_{s<t} \omega(s)  $, $t\in T$, $\omega  \in \Omega:=\{-1,1\}^T$,
which is equipped with the internal counting measure,
$x(b)$ is a lifting of $X$, $\ex$ is the expectation operator with respect to the internal counting measure, and the circle denotes the standard part.
\end{example}

\section{Strong $L^2$-approximation of the Skorokhod integral and the Malliavin derivative}
\label{sec:malliavin2}

In this section, we apply the Wiener chaos limit theorem (Theorem \ref{thm_S_transform_convergence_2}) in order to prove 
strong $L^2$-approximation results
for the Skorokhod integral and the Malliavin derivative. For the construction of the approximating sequences we compose the discrete Skorokhod 
integral and the discretized Malliavin derivative with the orthogonal projection on $\mathcal{H}^n$, i.e. on the subspace 
of random variables which admit a discrete
chaos decomposition in terms of multiple integrals with respect to the discrete time noise $(\xi^n_i)_{i\in \N}$.

We first treat the Malliavin derivative and aim at proving the following result.
\begin{theorem}\label{thm_Malliavin_derivative_convergence}
 Suppose $(X^n)_{n\in \N}$ converges strongly in $L^2(\Omega,\F,P)$ to $X$ and, for every $n\in \N$, $\pi_{\mathcal{H}^n}X^n \in \D^{1,2}_n$.
 Then the following are equivalent:
 \begin{enumerate}
\item[(i)] $\lim\limits_{m\rightarrow \infty}\limsup\limits_{n \rightarrow \infty}\sum\limits_{k=m}^{\infty} k k! \| \widehat{f^{n,k}_{X_n}} \|_{L^2([0,\infty)^k)}^2 =0$
(with $ \widehat{f^{n,k}_{X_n}}$ as defined in (\ref{eq_thm_S_transform_convergence_2_integrand})).
\item[(ii)] $X\in \mathbb{D}^{1,2}$ and 
the sequence $(D^n_{\left\lceil n \cdot\right\rceil}(\pi_{\mathcal{H}^n}X^n))_{n\in \N}$ converges to $DX$
strongly in $L^2(\Omega \times [0,\infty))$ as $n$ tends to infinity.
\end{enumerate}
\end{theorem}

Note first, that by continuity of $D^n_i$ for a fixed time $i\in \N$, we get
\begin{align*}
D^n_i (\pi_{\mathcal{H}^n}X^n) &= \sum\limits_{|A|<\infty} \ex[X^n \Xi^n_{A}] D^n_i \Xi^n_{A}=   \sqrt{n} \sum\limits_{|A|<\infty;\; i\in A} \ex[X^n \Xi^n_{A}]\Xi_{A\setminus \{i\}}^n\\
&=\sqrt{n} \sum_{|B|<\infty; \, i\notin B} 
\ex[X^n \Xi^n_{B\cup\{i\}}]\Xi_{B}^n. 
\end{align*}

By the relation \eqref{eq_discrete_chaos_decomposition}--\eqref{eq_discrete_chaos_integrand} between Walsh decomposition and 
discrete chaos decomposition, this identity can be reformulated as
\begin{equation}\label{eq:discrete_Malliavin_chaos}
D^n_i (\pi_{\mathcal{H}^n}X^n) = \sum\limits_{k=1}^{\infty} k I^{n,k-1}(f_{X^n}^{n,k}(\cdot, i)). 
\end{equation}
Hence, the isometry for discrete multiple Wiener integrals (\ref{eq_discrete_multiple_Wiener_covariance}) implies 
\begin{equation}\label{eq:discrete_Malliavin_isometry}
 \frac{1}{n}\sum_{i=1}^\infty \ex\left[ \left|D^n_i (\pi_{\mathcal{H}^n}X^n)\right|^2 \right]=  \sum\limits_{k=1}^{\infty} k k! \| \widehat{f_{X^n}^{n,k}} \|_{L^2([0,\infty)^k)}^2,
\end{equation}
i.e.,
\begin{equation}\label{eq:discrete_Malliavin_domain}
 \pi_{\mathcal{H}^n}X^n \in \D^{1,2}_n \; \Leftrightarrow \;  \sum\limits_{k=1}^{\infty} k k! \| \widehat{f_{X^n}^{n,k}} \|_{L^2([0,\infty)^k)}^2 <\infty.
\end{equation}
This is in line with the characterization of the continuous Malliavin derivative in terms of the chaos decomposition,
see e.g. \cite{Nualart}, which we show to be equivalent to Definition \ref{prop:Malliavin_domain} in the Appendix:
\begin{equation}\label{eq:Malliavin_domain}
 X \in \D^{1,2} \; \Leftrightarrow \;  \sum\limits_{k=1}^{\infty} k k! \| f^k_X \|_{L^2([0,\infty)^k)}^2 <\infty,
\end{equation}
and, if this is the case,
\begin{equation}\label{eq:Malliavin_chaos}
D_tX = \sum\limits_{k=1}^{\infty} k I^{n,k-1}(f_X^k(\cdot, t)),\; \textnormal{a.e. }t\geq 0,\quad \int_0^\infty \ex[ (D_tX)^2] dt= \sum\limits_{k=1}^{\infty} k k! \| f^k_X \|_{L^2([0,\infty)^k)}^2.
\end{equation}
After these considerations on the connection between (discretized) Malliavin derivative and (discrete) chaos decomposition, the proof 
of Theorem \ref{thm_Malliavin_derivative_convergence} turns out to be rather straightforward.
\begin{proof}[Proof of Theorem \ref{thm_Malliavin_derivative_convergence}]
 By Theorem \ref{thm_S_transform_convergence_2} (in conjunction with Remark \ref{rem:projections}), we observe that, for every 
 $k\in \N_0$, $(\widehat{f^{n,k}_{X_n}})_{n\in\N}$ converges to $f_X^k$ strongly in $L^2([0,\infty)^k)$. Hence,
 by (\ref{eq:discrete_Malliavin_isometry}), (\ref{eq:Malliavin_domain}), and (\ref{eq:Malliavin_chaos}),
 \begin{eqnarray*}
  (i) & \Leftrightarrow & \lim_{n\rightarrow \infty} \sum\limits_{k=1}^{\infty} k k! \| \widehat{f_{X^n}^{n,k}} \|_{L^2([0,\infty)^k)}^2=\sum\limits_{k=1}^{\infty} k k! \| f^k_X \|_{L^2([0,\infty)^k)}^2 <\infty \\
  & \Leftrightarrow &  X\in \D^{1,2} \textnormal{ and } 
  \lim_{n\rightarrow \infty} \frac{1}{n}\sum_{i=1}^\infty \ex\left[ \left|D^n_i (\pi_{\mathcal{H}^n}X^n)\right|^2 \right]=\int_0^\infty \ex[ (D_tX)^2] dt.
 \end{eqnarray*}
Hence, the asserted equivalence is a direct consequence of Theorem \ref{thm_Malliavin_derivative_convergence_1}. 
\end{proof}

We now wish to derive an analogous strong approximation result for the Skorokhod integral, which requires some additional notation.
For every $Z^n\in L^2_n(\Omega\times\N)$ and $k\in \N_0$, we denote
 $$
 \mathfrak f^{n,k}_{Z^n}(i_1,\ldots,i_k,i):=f^{n,k}_{Z^n_i}(i_1,\ldots,i_k)=\left\{ \begin{array}{cl} \ex\left[\frac{n^{k/2}}{k!} Z^n_i \Xi^n_{\{i_1, \ldots, i_k\}}\right], & |\{i_1, \ldots, i_k\}\cap \N|=k \\ 
0, & \textnormal{otherwise}. \end{array}\right.
 $$ 
  Then, with $\pi_{\mathcal{H}^n}Z^n:=(\pi_{\mathcal{H}^n}Z^n_i)_{i\in\N}$, 
$$
\sum_{k=0}^\infty k! \|\mathfrak f^{n,k}_{Z^n}\|^2_{L^2_n(\N^{k+1})}=\| \pi_{\mathcal{H}^n} Z^n\|^2_{L^2_n(\Omega\times \N)}<\infty,
$$
but $\mathfrak f^{n,k}_{Z^n}$ is symmetric in the first $k$ variables only and does not, in general, vanish on the diagonal.
For a function $F$ in $k$ variables, we denote its symmetrization by
$$
\widetilde F(y_1,\ldots,y_k)=\frac{1}{k!} \sum_{\pi} F(y_{\pi(1)},\ldots,y_{\pi(k)}),
$$
where the sum runs over the group of permutations of $\{1,\ldots,k\}$.
With this notation, $\widetilde{\mathfrak f}^{n,k}_{Z^n} \eins_{\partial_{k+1}^c}$ is an element of $\widetilde{L^2_n}(\N^{k+1})$.

We can now state:
\begin{theorem}\label{thm:Skorokhod_convergence_2}
Suppose that, for every $n\in \N$,  $Z^n\in L^2_n(\Omega\times\N)$ and  $\pi_{\mathcal{H}^n}Z^n\in D(\delta^n)$. Moreover,
assume that $(Z^n_{\lceil n\cdot\rceil})_{n\in\N}$ converges to $Z$ strongly in $L^2(\Omega\times [0,\infty))$.
Then, the following assertions are equivalent:
\begin{itemize}
 \item[(i)]   $\lim\limits_{m\rightarrow \infty}\limsup\limits_{n \rightarrow \infty}\sum\limits_{k=m}^{\infty} k! \| \widetilde{\mathfrak f}^{n,k-1}_{Z^n}\eins_{\partial_k^c}\|^2_{L^2_n(\N^k)} =0$.
 \item[(ii)] $Z\in D(\delta)$ and  $(\delta^n(\pi_{\mathcal{H}^n}Z^n))$ converges  to $\delta(Z)$ strongly in $L^2(\Omega,\F,P)$ as $n$ tends to infinity.
 \end{itemize}
\end{theorem}

As a preparation of the proof we note that, for every $M\in \N$,
\begin{eqnarray*}
 && \sum_{i=1}^M \ex[\pi_{\mathcal{H}^n}Z^n_i|\F^n_{M,-i}] \frac{\xi^n_i}{\sqrt{n}}= 
 \sum_{i=1}^M \sum_{|A|<\infty} \ex[Z^n_i\Xi^n_A] \,\ex[ \Xi^n_A|\F^n_{M,-i}] \frac{\xi^n_i}{\sqrt{n}}
 \\ &=& n^{-1/2} \sum_{i=1}^M \sum_{A\subset \{1,\ldots,M\}} \eins_{\{i\notin A\}} \, \ex[Z^n_i\Xi^n_A] \Xi^n_{A\cup\{i\}} 
 =n^{-1/2} \sum_{k=1}^M \sum_{B\subset\{1\ldots,M\}, |B|=k} \sum_{i\in B}\ex[Z^n_i \Xi^n_{B\setminus\{i\}}]\Xi^n_B \\
 &=& n^{-1/2} \sum_{k=1}^M  k \sum_{\substack{(i_1, \ldots, i_k) \in \N^k, \ i_1<\cdots<i_k}}  \eins_{[1,M]}^{\otimes k}(i_1,\ldots,i_k)
 \frac{1}{k}\sum_{j=1}^k \ex[Z^n_{i_j}\Xi^n_{\{i_1,\ldots,i_k\}\setminus\{i_j\}}]\Xi^n_{\{i_1,\ldots,i_k\}}\\
 &=& \sum_{k=1}^M I^{n,k}(\widetilde{\mathfrak f}^{n,k-1}_{Z^n}  \eins_{[1,M]}^{\otimes k} \eins_{\partial_{k}^c}).
\end{eqnarray*}

Hence, by the isometry for discrete multiple Wiener integrals, 
\begin{equation}
 \pi_{\mathcal{H}^n}Z^n \in D(\delta^n) \, \Leftrightarrow \, \sum_{k=1}^\infty k! \| \widetilde{\mathfrak f}^{n,k-1}_{Z^n} \eins_{\partial_{k}^c} \|^2_{L^2_n(\N^k)}<\infty,
\end{equation}
and, if this is the case,
\begin{equation}\label{eq:Skorokhod_discrete_chaos}
  \delta^n(\pi_{\mathcal{H}^n}Z^n)=\sum_{k=1}^\infty I^{n,k}(\widetilde{\mathfrak f}^{n,k-1}_{Z^n} \eins_{\partial_{k}^c}),
\end{equation}
i.e., $f^{n,0}_{\delta^n(\pi_{\mathcal{H}^n}Z^n)}=0$ and, for every $k\in \N$,
$$
f^{n,k}_{\delta^n(\pi_{\mathcal{H}^n}Z^n)} =\widetilde{\mathfrak f}^{n,k-1}_{Z^n} \eins_{\partial_{k}^c}.
$$

For the proof of Theorem \ref{thm:Skorokhod_convergence_2}, we also provide the following variant of 
Theorem \ref{thm_S_transform_convergence_2}, `$(i)\Rightarrow (ii)$',
for stochastic processes.

\begin{proposition}\label{thm_process_convergence_implies_chaos_convergences}
Suppose $Z^n \in L^2_n(\Omega \times \N)$ for every $n\in \N$  and $(Z^n_{\left\lceil n \cdot\right\rceil})$ converges strongly in $L^2(\Omega \times [0,\infty))$ to $Z$  as $n$ tends to infinity. 
Define the functions $\mathfrak{f}_Z^k\in L^2([0,\infty)^{k+1})$ via $\mathfrak{f}_Z^k(t_1,\ldots,t_{k+1}):=f_{Z_{t_{k+1}}}^k(t_1,\ldots,t_k)$. Then, for every $k\in \N_0$, as $n$ tends to infinity,
\begin{equation*}
\widehat{{\mathfrak f}^{n,k}_{Z^n}} \rightarrow  \mathfrak{f}_Z^k
\end{equation*}
strongly in $L^2([0,\infty)^{k+1})$.
\end{proposition}

\begin{proof}
The proof largely follows the arguments  in the proof of Theorem \ref{thm_S_transform_convergence_2}. We spell it out for sake of completeness.
 Let $g,h \in \mathcal{E}$. Then, by the isometry for (discrete) multiple Wiener integrals, Corollary \ref{cor_observation_Wick_powers_simple_functions},
and (\ref{eq:hilf0001}),
\begin{eqnarray}\label{eq:hilf0010}
 &&\left\langle \widehat{\mathfrak{f}^{n,k}_{Z^n}}, \widehat{(\check{g}^n)}^{{\otimes} k} \otimes  \widehat{\check{h}^n}\right\rangle_{L^2([0,\infty)^{k+1})}
= \frac{1}{n}\sum_{i=1}^{\infty} \left\langle {f}^{n,k}_{Z^n_i}, {(\check{g}^n)}^{{\otimes} k}\eins_{\partial_k^c}\right\rangle_{L^2_n(\N^{k})}  \check{h}^n(i) \nonumber\\ &=& 
\frac{1}{n}\sum_{i=1}^{\infty} \ex[(\pi_{\mathcal{H}^n}Z^n_i) I^{n,k}((\check{g}^n)^{\otimes k}\eins_{\partial_k^c})]\check{h}^n(i) =\frac{1}{n}\sum_{i=1}^{\infty} \ex[Z^n_i I^{n,k}((\check{g}^n)^{\otimes k}\eins_{\partial_k^c})]\check{h}^n(i) 
 \nonumber \\ &\rightarrow& \int_0^\infty  \ex[Z_s I^{k}(g^{\otimes k})]h(s)ds= 
 \left\langle \mathfrak{f}^{k}_{Z}, g^{{\otimes} k} \otimes  h\right\rangle_{L^2([0,\infty)^{k+1})}.
\end{eqnarray}
As
\begin{equation}\label{eq:hilf0011}
\sup_{n\in \N} \left\| \widehat{\mathfrak{f}^{n,k}_{Z^n}} \right\|^2_{{L^2([0,\infty)^{k+1})}}
=\sup_{n\in \N} \int_0^\infty \ex\left[ \left|I^{n,k}(f^{n,k}_{Z^n_{\left\lceil n s\right\rceil}} )\right|^2 \right]ds \leq \sup_{n\in\N}\left\| Z^n_{\left\lceil n \cdot\right\rceil}\right\|^2_{L^2(\Omega\times[0,\infty))}<\infty,
\end{equation}
$ \widehat{(\check{g}^n)}^{{\otimes} k} \otimes  \widehat{\check{h}^n}\rightarrow  g^{{\otimes} k} \otimes  h $ 
strongly in $L^2([0,\infty)^{k+1})$ by (\ref{eq:hilf0001}), and the set
$\{g^{{\otimes} k} \otimes h : g,h \in \mathcal{E}\}$ is total in the closed subspace of functions in 
$L^2([0,\infty)^{k+1})$, which are symmetric in the first $k$ variables, we conclude again that $\widehat{\mathfrak{f}^{n,k}_{Z^n}}$ 
converges weakly to $\mathfrak{f}^{k}_{Z}$ in this subspace.
Hence, it only remains to argue that
\begin{equation*}
 \left\| \widehat{\mathfrak{f}^{n,k}_{Z^n}}\right\|^2_{L^2([0,\infty)^{k+1})}\rightarrow \left\| \mathfrak{f}^{k}_{Z}\right\|^2_{L^2([0,\infty)^{k+1})},\quad n\rightarrow \infty.
\end{equation*}
As
\begin{eqnarray*}
\frac{1}{n}\sum_{i=1}^{\infty} \ex[Z^n_i I^{n,k}((\check{g}^n)^{\otimes k}\eins_{\partial_k^c})]\check{h}^n(i) &=&
\frac{1}{n}\sum_{i=1}^{\infty} (S^nI^{n,k}(f^{n,k}_{Z^n_i} ))(\check g^n)\check{h}^n(i), \\
 \int_0^\infty  \ex[Z_s I^{k}(g^{\otimes k})]h(s)ds&=& \int_0^\infty  (S I^k(f^k_{Z_s}))(g)h(s)ds,
\end{eqnarray*}
we may derive from (\ref{eq:hilf0010})--(\ref{eq:hilf0011}) and  Theorem \ref{thm_S_transform_process_convergence_discrete}, 
that  $I^{n,k}(f^{n,k}_{Z^n_{\left\lceil n \cdot\right\rceil}})$ converges to $I^k(f^k_{Z_{\cdot}})$ weakly in $L^2(\Omega\times [0,\infty))$.
Thus,
\begin{eqnarray*}
  \left\| \widehat{\mathfrak{f}^{n,k}_{Z^n}}\right\|^2_{L^2([0,\infty)^{k+1})}&=& 
  \int_0^\infty \ex\left[ I^{n,k}(f^{n,k}_{Z^n_{\left\lceil n s\right\rceil}}) Z_s \right]ds + 
   \int_0^\infty \ex\left[ I^{n,k}(f^{n,k}_{Z^n_{\left\lceil n s\right\rceil}}) (Z^n_{\left\lceil n s\right\rceil}-Z_s) \right]ds \\
   &\rightarrow&  \int_0^\infty \ex\left[ I^{k}(f^{k}_{Z_s}) Z_s \right]ds =\left\| \mathfrak{f}^{k}_{Z}\right\|^2_{L^2([0,\infty)^{k+1})}.
\end{eqnarray*}
\end{proof}

\begin{proof}[Proof of Theorem \ref{thm:Skorokhod_convergence_2}]
By the linearity of the embedding operator $\widehat{(\cdot)}$, Minkowski inequality, Proposition \ref{thm_process_convergence_implies_chaos_convergences}, and Lemma \ref{lem:diagonal}, we obtain,
for every $k\in \N_0$,
\begin{eqnarray*}
 \left\| \widehat{\widetilde{\mathfrak f}^{n,k}_{Z^n}\eins_{\partial_{k+1}^c}}-\widetilde{\mathfrak f}^{k}_Z\right\|_{L^2([0,\infty)^{k+1})}
 = \left\| \widetilde{\widehat{\mathfrak f}^{n,k}_{Z^n}\eins_{\partial_{k+1}^c}}-\widetilde{\mathfrak f}^{k}_Z\right\|_{L^2([0,\infty)^{k+1})}
  \leq
 \left\| \widehat{\mathfrak f^{n,k}_{Z^n}\eins_{\partial_{k+1}^c}}-\mathfrak f^{k}_Z\right\|_{L^2([0,\infty)^{k+1})}\rightarrow 0
\end{eqnarray*}
as $n$ tends to infinity. Thus, due to Theorem \ref{thm_S_transform_convergence_2} and (\ref{eq:Skorokhod_discrete_chaos}),
\begin{eqnarray*}
 (i) \, \Leftrightarrow \,  \left( \delta^n(\pi_{\mathcal{H}^n}Z^n) \right)_{n\in \N} \textnormal{ converges strongly in } L^2(\Omega,\F,P).
\end{eqnarray*}
Now, the implication `$(ii) \Rightarrow (i)$' is obvious, while the converse implication is a consequence 
of Theorem \ref{thm:Skorokhod_convergence_1}.
\end{proof}

\begin{remark}
 As a by-product of the proof of Theorem \ref{thm:Skorokhod_convergence_2}, we  recover, thanks to Theorem \ref{thm_S_transform_convergence_2},
 the well-known chaos decomposition of the Skorokhod integral as
 $$
 \delta(Z)=\sum_{k=1}^\infty I^{k}( \widetilde{\mathfrak f}^{k-1}_{Z}).
 $$
\end{remark}

\section{Binary noise}\label{section_binary}

In this section, we specialize to the case of binary noise, i.e., we suppose that, for  some constant $b>0$, 
$$
P(\{\xi=-1/b\})=\frac{b^2}{b^2+1},\quad P(\{\xi=b\})=\frac{1}{b^2+1}.
$$
 We illustrate, that in this binary case, our approximation formulas for the Malliavin derivative and the Skorokhod integral
give rise to a straightforward numerical implementation.

We recall first that Malliavin calculus on the Bernoulli space is well-studied, see, e.g. \cite{Holden_discrete_Wick}, \cite{Leitz-Martini}, 
\cite{Privault_book}, and the references 
therein, usually with the aim to explain the main ideas of Malliavin calculus  by discussing the analogous 
operators in the simple toy setting.
Note first that $L^2(\Omega,\F^n_i,P)$ equals $\mathcal{H}^n_i$ in the binary case (and in this case only) by observing that 
both spaces have dimension $2^i$ in this case. Hence, $L^2(\Omega,\F^n,P)$ coincides with $\mathcal{H}^n$ for binary noise, and we can 
drop the orthogonal projections $\pi_{\mathcal{H}^n}$ on $\mathcal{H}^n$ in the statement of all previous results. In particular, 
every random variable $X^n\in L^2(\Omega,\F^n,P)$ then admits a chaos decomposition in terms of discrete multiple Wiener integrals,
and the representations of the discretized Malliavin derivative and the discrete Skorokhod integral in terms of the discrete chaos 
in Section \ref{sec:malliavin2} show that these operators coincide with the Malliavin derivative and the Skorokhod integral on 
the Bernoulli space, see \cite{Privault_book}.

In the binary case, the representations for the discrete Mallivain derivative and the discrete Skorokhod integral can be simplified 
considerably. Suppose $X^n\in L^2(\Omega,\F^n,P)$. Then, there is a measurable map $F_{X^n}: \mathbb{R}^\infty \rightarrow \mathbb{R}$ such that
$X^n=F_{X^n}(\xi^n_1,\xi^n_2,\ldots)$. A direct computation shows that, for every $i\in \N$,
\begin{eqnarray}\label{eq:Malliavin_binary}
\nonumber D^n_i X&=& \sqrt{n} \ex[\xi^n_i  F_{X^n}(\xi^n_1,\xi^n_2,\ldots)|\F^n_{-i}] \\ &=& \frac{\sqrt{n} b}{b^2+1} \left( F_{X^n}(\xi^n_1,\ldots,\xi^n_{i-1},b,\xi^n_{i+1},\ldots)-
 F_{X^n}(\xi^n_1,\ldots,\xi^n_{i-1},-1/b,\xi^n_{i+1},\ldots) \right),
\end{eqnarray}
hence, the Malliavin derivative becomes a difference operator. Moreover, for $Z^n\in L^2_n(\Omega\times \N)$
and $N\in \N$, the discrete Skorokhod integral 
can be rewritten as
\begin{eqnarray*}
 \delta^n(Z^n \eins_{[1,N]}) = \sum_{i=1}^N Z^n_i \frac{\xi^n_i}{\sqrt n} -\frac{1}{n}\sum_{i=1}^N (\xi^n_i)^2 D^n_i Z^n_i,
\end{eqnarray*}
which can either be derived from \cite[Proposition 1.8.3]{Privault_book} or by expanding $Z^n_i$ in its Walsh decomposition and noting that, 
for every finite subset $A\subset \N$,
$$
\left(\Xi^n_A- \ex[\Xi^n_A|\F^n_{-i}]\right)\sqrt{n} \xi^n_i =  \left\{\begin{array}{cl} \sqrt{n} \Xi^n_{A\setminus\{i\}} (\xi^n_i)^2, & i\in A \\
 0, & i \notin A \end{array} \right\} =(\xi^n_i)^2 D^n_i \Xi^n_A.
$$
Hence, for  $Z^n\in L^2_n(\Omega\times \N)$ and $N\in \N$,
\begin{align}\label{eq:Skorokhod_binary}
\delta^n(Z^n\eins_{[1,N]} ) &=  \sum_{i=1}^N F_{Z^n_i}(\xi^n_1,\xi^n_2,\ldots) \frac{\xi^n_i}{\sqrt n} 
 \nonumber\\ 
 &-\frac{ (\xi^n_i)^2 b}{\sqrt{n}(b^2+1)} \left( F_{Z^n_i}(\xi^n_1,\ldots,\xi^n_{i-1},b,\xi^n_{i+1},\ldots)-
 F_{Z^n_i}(\xi^n_1,\ldots,\xi^n_{i-1},-1/b,\xi^n_{i+1},\ldots) \right).  
\end{align}

Recall that the discrete noise $(\xi^n_i)_{i\in \N}$, can be constructed from the underlying Brownian motion $(B_t)_{t\in [0,\infty)}$
via a Skorokhod embedding as
$$
\xi^n_i=\sqrt{n} \left(B_{\tau^n_i}- B_{\tau_{i-1}^n}\right),
$$
where, in the binary case,
\begin{equation}\label{eq:embedding_binary}
\tau^n_0:=0 \ , \  \tau^n_i := \inf\left\{s \geq \tau^n_{i-1} : B_{s}-B_{\tau^n_{i-1}} \in \left\{\frac{b}{\sqrt{n}}, \frac{-1}{b \sqrt{n}} \right\}\right\}, 
\end{equation}
and the Brownian motion at the first-passage times $\tau^n_i$ can be simulated by the acceptance-rejection algorithm of \cite{Burq_Jones}.

We close this paper by a toy example which illustrates how to numerically compute Skorokhod integrals by our approximation results.
\begin{example}
 In this example, we approximate the Skorokhod integral $\delta(Z)$
 for the process
 $$
 Z_t= \sign(1/2-t)(B_1 B_{1-t} - (1-t)))\eins_{[0,1]}(t), \quad t\geq 0,
 $$
 where we choose the sign-function to be rightcontinuous at 0. For the discrete time approximation we consider 
 $$
 Z^n_i=  \sign(1/2-i/n)\left(B^n_n B^n_{n-i} -  (1-i/n)\right)\eins_{[1,n-1]}(i),\quad i\in \N,
 $$
 and note that $(Z^n_{\lceil nt \rceil})$ converges to $Z_t$ for almost every $t\geq 0$ in probability by (\ref{eq_Donsker_embedding}).
Hence, by
 uniform integrability and dominated convergence, it is easy to check that $(Z^n_{\lceil n \cdot \rceil})_{n\in \N}$ converges to $Z$ strongly in $L^2(\Omega\times [0,\infty))$. 
 \begin{figure}[ht]
\centering
\includegraphics[scale=0.75]{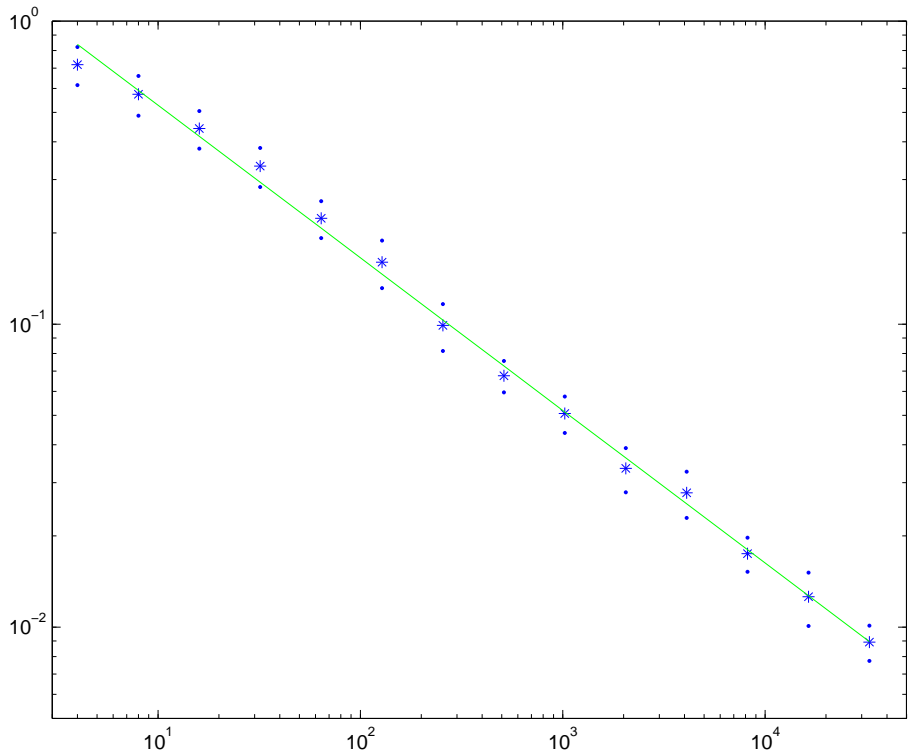}
\vspace*{-1cm} \caption{Log-log plot of the simulated strong $L^2(\Omega,\F,P)$-approximation as the number of time steps increases.}\label{fig1}
\end{figure}
 We next observe that  in the discrete chaos decomposition of  $\delta^n(Z^n)$, all the coefficient functions 
 $f_{\delta^n(Z^n)}^{n,k}$ for $k\geq 4$ vanish, because $Z^n_i$ is a polynomial of degree 2 in $B^n$. 
 Hence, the tail condition 
 in Theorem \ref{thm:Skorokhod_convergence_2} is trivially satisfied and, consequently, $(\delta^n(Z^n))_{n\in \N}$ converges 
 to $\delta(Z)$ strongly in $L^2(\Omega,\F,P)$. We now suppose that $B^n$ is constructed via the Skorokhod embedding (\ref{eq:embedding_binary})
 and simulate, for $n=4,8,\ldots, 2^{15}$, 10000 independent copies $(B^{n,l})_{l=1,\ldots, 10000}$  of $B^n$  by the Burq\&Jones algorithm. 
 The correponding realizations of $\delta^n(Z^n)$ and $\delta(Z)$ along the $l$th trajectory of the underlying Brownian motion 
 are denoted $\delta^{n}_l(Z^n)$ and $\delta_l(Z)$, $l=1,\ldots 10000$, respectively. For the discrete Skorokhod integral 
 we implement formula (\ref{eq:Skorokhod_binary}) with $N=n$, while for the continuous Skorokhod integral we exploit that it can be computed analytically 
 and equals
 $$
 \delta(Z)= B_1 B_{1/2}^{2}-\frac{B_1}{2} - B_{1/2}.
 $$
 Figure \ref{fig1} shows, in the case of symmetric binary noise ($b=1$), a log-log-plot of the empirical mean (indicated by crosses) of 
  $|\delta^n_l(Z^n)-\delta_l(Z)|^2$, $l=1,\ldots, 10000$, and the corresponding (asymptotical) 95\%-confidence bounds  (indicated by dots)
 as the number of time steps $n$ increases. A linear regression (solid line) exhibits a slope of $-0.5036$ and, thus, indicates that strong $L^2(\Omega,\F,P)$-convergence
 takes place at the expected rate of $1/2$. 
\end{example}

\appendix

\section{$S$-transform characterization of the Malliavin derivative}

In this appendix, we prove the equivalence between the definition of the Malliavin derivative in terms of the $S$-transform 
(Definition \ref{prop:Malliavin_domain})  and the more classical characterization in terms of the chaos decomposition, see 
(\ref{eq:Malliavin_domain})--(\ref{eq:Malliavin_chaos}).

\begin{proposition}
 Suppose $X=\sum_k I^k(f^k_X) \in L^2(\Omega,\F,P)$. Then, the following are equivalent:
 \begin{enumerate}
  \item[(i)] There 
 is a stochastic process $Z\in L^2(\Omega\times [0,\infty))$ such that for every $g,h\in \mathcal{E}$,
 $$
 \int_0^\infty  (SZ_s)(g) h(s) ds =\ex\left[X \exp^{\diamond}(I(g)) \left(I(h) - \int_0^\infty  g(s)h(s) ds \right) \right].
 $$
  \item[(ii)] $\sum\limits_{k=1}^{\infty} k k! \| f^k_X \|_{L^2([0,\infty)^k)}^2 <\infty$.
 \end{enumerate}
If this is the case, then $Z_t=\sum\limits_{k=1}^{\infty} k I^{n,k-1}(f_X^k(\cdot, t))$ for almost every $t\geq 0$.
\end{proposition}
\begin{proof}
We first note that, for every $f,g\in \mathcal{E}$, 
\begin{equation}\label{eq:a1}
 \exp^{\diamond}(I(g)) \left(I(h) - \int_0^\infty  g(s)h(s) ds \right)= \sum_{k=1}^\infty \frac{1}{(k-1)!} I^k(\widetilde{(g^{\otimes(k-1)} \otimes h})),
\end{equation}
which can be verified by computing the $S$-transform of both sides.
By the Cauchy-Schwarz inequality, we obtain for every $f,g\in \mathcal{E}$,
\begin{eqnarray}\label{eq:hilf0005}
\nonumber && \sum_{k=1}^\infty \int_{[0,\infty)^k} \left|k f_{X}^{k}(x)
({g}^{{\otimes} k-1}{\otimes} {h})(x) \right|dx \\
 &\leq & \left(\sum_{k=1}^\infty k! \| {f_{X}^{k}}\|^2_{L^2([0,\infty)^{k})}\right)^{1/2}
 \left(\sum_{k=1}^\infty \frac{k}{(k-1)!}  \| g \|^{2(k-1)}_{{L^2}([0,\infty))} \| h \|^{2}_{{L^2}([0,\infty))}\right)^{1/2}<\infty.
\end{eqnarray}
Hence, Fubini's theorem implies
\begin{equation*}
  \sum_{k=1}^\infty \int_{[0,\infty)^k} k f_{X}^{k}(x) (g^{{\otimes} k-1} \otimes {h})(x) dx 
  = \int_0^\infty \left(\sum_{k=1}^\infty \int_{[0,\infty)^{k-1}} k f_{X}^{k}(x,t) g^{{\otimes} k-1}(x)\right) h(t) dt,
 \end{equation*}
 i.e., by (\ref{eq:a1}) and the isometry for multiple Wiener integrals,
 \begin{equation}\label{eq:hilf0004}
  \ex\left[X \exp^{\diamond}(I(g)) \left(I(h) - \int_0^\infty  g(s)h(s) ds \right) \right]=\int_0^\infty \left(\sum_{k=1}^\infty \int_{[0,\infty)^{k-1}} k f_{X}^{k}(x,t) g^{{\otimes} k-1}(x)\right) h(t) dt
 \end{equation}
 for every $g,h\in \mathcal{E}$.
\\[0.2cm]
`$(i) \Rightarrow (ii)$':
Assuming (i) and noting that (\ref{eq:hilf0004}) holds for every $g,h\in \mathcal{E}$, we observe that
for every $g\in \mathcal{E}$, $\alpha\in\R$, and Lebesgue-almost every $s\in [0,\infty)$,
$$
 \sum_{k=1}^\infty \alpha^{k-1} \langle f^{k-1}_{Z_s}(\cdot), g^{\otimes (k-1)} \rangle_{L^2([0,\infty)^{k-1})} =  (S Z_s)(\alpha g) = \sum_{k=1}^\infty  \alpha^{k-1}\langle k f_X^k(\cdot,s), g^{\otimes (k-1)} \rangle_{L^2([0,\infty)^{k-1}}. 
$$
(Note, that the Lebesgue null set can be chosen independent of $g$, $\alpha$. Indeed, one can first take $\alpha \in \mathbb{Q}$ and
step functions $g$ with rational step sizes and interval limits, and then pass to the limit). 
Comparing the coefficients in the power series and noting that $\{g^{\otimes k}, g\in \mathcal{E}\}$ is total in  $\widetilde{L^2}([0,\infty)^{k})$, we 
obtain, for every $k\geq 1$ and almost every $s\in [0,\infty)$,
\begin{equation}\label{eq:hilf0008}
k f_{X}^{k}(\cdot,s)=f^{k-1}_{Z_s}.
\end{equation}
Therefore, the isometry for multiple Wiener-It\^o integrals implies                                                                                                          
\begin{equation}\label{eq:a2}
\sum\limits_{k=1}^{\infty} k k! \| f_X^k \|_{L^2([0,\infty)^k)}^2 = \int_0^\infty \ex[|Z_s|^2] ds<\infty.
\end{equation}
`$(ii) \Rightarrow (i)$': Define
$
Z_t=\sum\limits_{k=1}^{\infty} k I^{n,k-1}(f_X^k(\cdot, t)).
$
Assuming (ii), we observe by the first identity in (\ref{eq:a2}) that $Z$ belongs to $L^2(\Omega\times[0,\infty))$. By the isometry 
for multiple Wiener integrals and the chaos decomposition of a Wick exponential we get, for every $g,h\in \mathcal{E}$.
$$
\int_0^\infty  (SZ_s)(g) h(s) ds= \int_0^\infty \left( \sum\limits_{k=1}^{\infty} \int_{[0,\infty)^{k-1}} k f_X^k(x, t) g^{\otimes k-1}(x) dx \right) h(t)dt.
$$
Hence, (\ref{eq:hilf0004}) concludes. 
\end{proof}


\begin{thebibliography}{}












\bibitem[Aigner (2007)]{Aigner} Aigner, M. {\it A course in Enumeration.} Berlin, Heidelberg: Springer (2007).

\bibitem[Avram (1988)]{Av} Avram, F. Weak convergence of the variations, iterated integrals and Doléans-Dade exponentials of sequences of semimartingales.
{\it Ann. Probab.} {\bf 16} (1), 246--250 (1988).

\bibitem[Avram and Taqqu (1986)]{AT}
Avram, F. and  Taqqu, M.
Symmetric polynomials of random variables attracted to an infinitely divisible law.
{\it Probab. Theory Relat. Fields} {\bf 71} (4), 491--500 (1986). 

\bibitem[Bai and Taqqu (2014)]{BT} Bai, S. and Taqqu, M.    Generalized Hermite processes, discrete chaos and limit theorems. {\it Stochastic
Processes Appl.} {\bf 124} (4),  1710--1739 (2014).








\bibitem[Briand et al. (2002)]{Briand_Delyon_Memin} Briand, P. and  Delyon, B. and M\'emin, J.
On the robustness of backward stochastic differential equations. {\it Stochastic Processes Appl.} 
{\bf 97}, 229--253, (2002).


\bibitem[Burq and Jones (2008)]{Burq_Jones} Burq, Z. A. and Jones, O. D., Simulation of Brownian motion at first-passage times. 
{\it Math. Comput. Simulation} {\bf 77}, 64--71, (2008).


\bibitem[Coquet et al. (2001)]{CMS}
Coquet, F., M\'emin, J., and Slominski, L. On weak convergence of filtrations. {\it S\'eminaire de probabilit\'es} {\bf 35}, 
306--328, (2001).

\bibitem[Cutland and Ng (1991)]{Cutland_Ng} Cutland, N. and Ng, S.
On homogeneous chaos.
\textit{Math. Proc. Cambridge Philos. Soc.} \textbf{110} (2), 353--363, (1991). 

\bibitem[Geiss et al. (2012)]{GG}
Geiss, C. and  Geiss, S. and Gobet, E.
Generalized fractional smoothness and $L^p$-variation of BSDEs with non-Lipschitz terminal condition. 
{\it Stochastic Process. Appl.} {\bf 122} (5) 2078--2116 (2012). 




\bibitem[Gzyl (2006)]{Gzyl} Gzyl, H.
An expos{\'e} on discrete Wiener chaos expansions.
\textit{Bol. Asoc. Mat. Venez.} \textbf{13} (1), 3--27, (2006).




\bibitem[Holden et al. (1992)]{Holden_discrete_Wick} Holden, H. and Lindstr{\o}m, T. and {\O}ksendal, B. and Ub{\o}e, J. 
Discrete Wick products. Stochastic analysis and related topics (Oslo, 1992),  Stochastics Monogr., \textbf{8}, Gordon and Breach, Montreux, 123--148, (1993). 


\bibitem[Holden et al. (2010)]{Holden_Buch} Holden H. and {\O}ksendal, B. and Ub{\o}e, J. and Zhang, T.
\textit{Stochastic Partial Differential Equations. A Modeling, White Noise
    Functional Approach. Second Edition} New York: Springer (2010).



\bibitem[Jacod et al. (2000)]{JaMP}
Jacod, J. and M\'el\'eard, S. and  Protter, P. 
Explicit form and robustness of martingale representations. 
{\it Ann. Probab.}  {\bf 28} (4), 1747--1780 (2000). 


\bibitem[Jakubowski et al. (1989)]{JMP}
Jakubowski, A. and  M\'emin, J. and  Pag\`es, G.
Convergence en loi des suites d'int\'egrales stochastiques sur l'espace $D^1$ de Skorokhod.
 {\it Probab. Theory Related Fields} {\bf 81} (1), 111--137 (1989). 
 
 
\bibitem[Janson (1997)]{Janson} Janson, S.
\textit{Gaussian Hilbert Spaces.}, Cambridge: Cambridge University Press (1997).


\bibitem[Krakowiak and Szulga (1986)]{KS} Krakowiak, W. and Szulga, J. {\it Random multilinear forms.}
{\it Ann. Probab.} {\bf 14} (3), 955--973,
(1986).

\bibitem[Kuo (1996)]{Kuo} Kuo, H.-H.
\textit{White Noise Distribution Theory.}
Probability and Stochastics Series. Boca Raton, FL: CRC Press (1996).

\bibitem[Kurtz and Protter (1991)]{KP} Kurtz, T. G. and Protter, P. Weak limit theorems for stochastic integrals and stochastic differential equations. {\it Ann. Probab.} {\bf 19} (3), 1035--1080 (1991). 


\bibitem[Le\~ao and Ohashi (2013)]{Leao_Ohashi}
Le\~ao, D. and Ohashi, A. Weak approximation for Wiener functionals. {\it Ann. Appl. Probab.} {\bf 23} (4), 1660--1691 (2013).

\bibitem[Leitz-Martini (2000)]{Leitz-Martini} Leitz-Martini, M.
A discrete Clark-Ocone formula. Maphysto Research Report No 29 (2000).









\bibitem[M\"orters and Peres (2010)]{Moerters_Peres} M\"orters, P. and Peres, Y. \textit{Brownian motion} Cambridge University Press, Cambridge, (2010).




\bibitem[Nualart (2006)]{Nualart} Nualart, D. 
\textit{The Malliavin Calculus and Related Topics.} Second Edition. Probability and its Applications (New York). Springer (2006). 

\bibitem[Nualart and Pardoux (1988)]{NualartPardoux}
Nualart, D. and Pardoux, E. Stochastic calculus with anticipating integrands. \textit{ Probab. Theory Related Fields} 
{\bf 78} (4), 
535--581 
(1988) 








\bibitem[Privault (2009)]{Privault_book} Privault, N.
\textit{Stochastic Analysis in Discrete and Continuous Settings.} Lecture Notes in Mathematics 1982. Berlin: Springer (2009).

\bibitem[Shiryaev (1996)]{Sh} Shiryaev, A. N. {\it Probability}. Second edition. Graduate Texts in Mathematics 95. Berlin: Springer (1996).

\bibitem[Sottinen (2001)]{Sottinen} Sottinen, T.
Fractional Brownian motion, random walks and binary market models.
\textit{Finance Stoch.} \textbf{5}, 343--355 (2001).


\bibitem[Surgailis (1982)]{Surgailis} Surgailis, D.
Domains of attraction of self-similar multiple integrals. {\it Lithuanian Math. J.} \textbf{22} (3), 327--340, (1982).



\bibitem[Yosida (1995)]{Yosida} Yosida, K. \textit{Functional analysis.} Reprint of the sixth (1980) edition. Classics in Mathematics. Berlin. Springer (1995).

\bibitem[Zhang (2004)]{Zhang} Zhang, J.
A numerical scheme for BSDEs. 
{\it Ann. Appl. Probab.} {\bf 14} (1), 459--488 (2004). 

\end{thebibliography}
\end{document}